\documentclass[11pt,a4paper,twoside,reqno]{amsart}
\usepackage[DIV=14,BCOR=15mm]{typearea}
\usepackage{amssymb,amsmath,amsfonts,amsthm,mathtools,microtype}
\allowdisplaybreaks
\usepackage{enumerate,expdlist,tikz,float}
  \usetikzlibrary{patterns}
\usepackage[english]{babel}
\usepackage[T1]{fontenc}
\usepackage[utf8]{inputenc}
\DeclarePairedDelimiter{\abs}{\lvert}{\rvert}
\DeclarePairedDelimiter{\norm}{\lVert}{\rVert}
\newcommand{\euler}{\mathrm{e}}

\newcommand{\sfuc}{\mathrm{uc}}
\newcommand{\drm}{\mathrm{d}}
\newcommand{\CC}{\mathbb{C}}
\newcommand{\NN}{\mathbb{N}}
\newcommand{\RR}{\mathbb{R}}
\newcommand{\TT}{\mathbb{T}} 
\newcommand{\ZZ}{\mathbb{Z}}
\newcommand{\supp}{\operatorname{supp}}

\newcommand{\ran}{\operatorname{Ran}}
\newcommand{\Ker}{\operatorname{Ker}}

\newcommand{\per}{\mathrm{per}}
\newcommand{\dir}{\mathrm{Dir}}
\newcommand{\obs}{\mathrm{obs}}
\newcommand{\ur}{\mathrm{ur}}
\newcommand{\equiset}{{S_{\delta , Z}}}

\newcommand{\cD}{\mathcal{D}}
\newcommand{\cB}{\mathcal{B}}
\newcommand{\cT}{\mathcal{T}}
\newcommand{\cH}{\mathcal{H}}

\newcommand{\cL}{\mathcal{L}}

\newcommand{\cF}{\mathcal{F}}
\newcommand{\cU}{\mathcal{U}}
\newcommand{\Id}{\mathrm{Id}}
\newcommand{\cO}{\mathcal{O}}

\newcommand{\adm}{\mathrm{Adm}}

\newcommand{\x}{z}
\newtheorem{theorem}{Theorem}[section]
\newtheorem{lemma}[theorem]{Lemma}

\newtheorem{corollary}[theorem]{Corollary}

\theoremstyle{definition}
\newtheorem{definition}[theorem]{Definition}
\theoremstyle{remark}
\newtheorem{remark}[theorem]{Remark}
\newtheorem{example}[theorem]{Example}
\usepackage[linktocpage, pdftex]{hyperref}
  \definecolor{darkred}{rgb}{0.5,0,0}
  \definecolor{darkgreen}{rgb}{0,0.5,0}
  \definecolor{darkblue}{rgb}{0,0,0.5}
  \hypersetup{colorlinks,linkcolor=darkblue,filecolor=darkgreen,urlcolor=darkred,citecolor=darkblue}
  \hypersetup{
  pdftitle={Null controllability and control cost estimates for the heat equation on unbounded and bounded domains},
  pdfauthor={M. Egidi, I. Nakic, A.~Seelmann, M. T\"aufer, M. Tautenhahn, I. Veseli\'c},
  }

\begin{document}
%
%
%
%
%
%
\title[Control cost estimates for the heat equation on unbounded domains]
{Null-controllability and control cost estimates for the heat equation on
unbounded and large bounded domains}

\author[M. Egidi]{Michela Egidi}
\author[I. Naki\'c]{Ivica Naki\'c}
\author[A. Seelmann]{Albrecht Seelmann}
\author[M. T\"aufer]{Matthias T\"aufer}
\author[M. Tautenhahn]{Martin Tautenhahn}
\author[I. Veseli\'c]{Ivan Veseli\'c}

\address[M.~Egidi]{Ruhr Universit\"at Bochum, Germany}
\address[I.~Naki\'c]{University of Zagreb, Croatia}
\address[A.~Seelmann, I.~Veseli\'c]{Technische Universit\"at Dortmund, Germany}
\address[M.~T\"aufer]{Queen Mary University of London, United Kingdom}
\address[M.~Tautenhahn]{Technische Universit\"at Chemnitz, Germany}
\keywords{Unique continuation, uncertainty principle, heat equation, null-controllabil\-ity, control costs, observability estimate, unbounded domains
}
\subjclass[2010]{35Pxx, 35J10, 35B05, 35B60, 81Q10}

\begin{abstract}
We survey recent results on the control problem for the heat equation on unbounded and large bounded domains. First we formulate new uncertainty relations, respectively spectral inequalities.
Then we present an abstract control cost estimate which improves upon earlier results. The latter is particularly interesting when combined with the earlier mentioned spectral inequalities since it yields sharp control cost bounds in several asymptotic regimes. We also show that control problems on unbounded domains can be approximated by corresponding problems on a sequence of bounded domains forming an exhaustion. Our results apply also for the generalized heat equation associated with a Schr\"odinger semigroup.
\end{abstract}
\maketitle
\vspace{-0.8cm}
\tableofcontents

%
%
%
%

\section{Introduction}

We survey several new results on the control problem of the heat equation on unbounded and large bounded domains.
The study of heat control on bounded domains has a long history, while unbounded domains
became a focus of interest only quite recently. In order to compare and interpolate these two geometric situations
it is natural to study the control problem on large bounded domains including a quantitative and explicit analysis
of the influence of the underlying geometry.
Here the term `large domain' may be made precise in at least two ways. For instance, it could mean that we study the control problem on a sequence of domains which form an exhaustion of the whole Euclidean space.
Alternatively, it could mean that the considered domain is large compared to some characteristic
length scale of the system, e.g.\ determined by the properties of the control/observability set.
Not surprisingly, the results on unbounded and large bounded domains which we present draw on concepts and methods which have been developed initially for control problems on
generic bounded domains.
While these previous results focused on giving precise criteria for (null-)controllability to hold, only a partial analysis of the influence of the underlying geometry  on
the control cost has been carried out.
Merely the dependence on the time interval length in which the control is allowed to take place has been studied thoroughly.
However,  recently there has been an increased interest in the role of geometry for the control cost.
We survey a number of recent results which perform a systematic analysis of the dependence of control cost estimates on
characteristic length scales of the control problem.
As a side benefit we obtain new qualitative results, most prominently a sharp,
i.e.\ sufficient and necessary, condition on the control/observability set which ensures
the null-controllability of the classical heat equation on the whole of $\RR^d$.

The results on null-controllability, in accordance with previous proofs, are obtained in two steps.
The first consists in some hard analysis and depends on the specific partial differential equation at hand
whereas the second one can be formulated in an abstract operator theoretic language. Let us discuss these two ingredients separately.

The mentioned hard analysis component of the proof consists in a variant of the \emph{uncertainty relation} or \emph{uncertainty principle}.
These terms stem from quantum physics and encode the phenomenon that the position and the momentum of a particle cannot be measured simultaneously
with arbitrary precision. Note that the momentum representation of an observable is obtained from the position representation via the Fourier transform.
Hence the fact that a non-trivial function and its Fourier transform cannot be simultaneously compactly supported is a particular manifestation of the uncertainty principle.
This qualitative theorem can be given a quantitative form  in various ways, e.g.~by the Paley-Wiener Theorem or the Logvinenko--Sereda Theorem which we
discuss in Section \ref{sec:UCP_Logvinenko-Sereda}.
If the property that a function has compactly supported Fourier transform is replaced by some similar restriction,
for instance that it is an element of a spectral subspace of a self-adjoint
Hamiltonian describing the total energy of the system, other variants of the uncertainty relation are obtained.
In the particular case that the Hamiltonian is represented by a second order elliptic partial differential operator with sufficiently regular coefficients
a particular instance of an uncertainty principle is embodied in (a quantitative version of) the \emph{unique continuation principle}.
The latter states that an eigenfunction (or, more generally a finite linear combination of eigenfunctions or elements from spectral subspaces associated to bounded energy intervals)
cannot vanish in the neighborhood of a point faster than a specified rate.
Such a quantitative unique continuation estimate in turn implies what is called a \emph{spectral  inequality}
in the context of control theory. This term was first coined for evolutions determined by the Laplace operator but is now used
also for abstract systems. Thus, it is hardly distinguishable from the notion of an uncertainty relation.
Note however that the term spectral inequality is used in other areas of mathematics with a  different meaning, e.g.~in Banach algebras or matrix analysis.

The second mentioned step uses operator theoretic methods and ODEs in Hilbert space to deduce
observability and controllability results from the hard analysis bound obtained in the first step.
There are several related but distinct approaches to implement this. One of them we present
in full for pedagogical reasons.
The other ones are not developed in this paper, but we discuss the resulting quantitative bounds on the control cost.
In fact, these seem to be better than what can be obtained by the mentioned pedagogical approach.

Let us point out several special features of this survey (and the underlying original research articles):
The uncertainty principles or spectral inequalities, and consequently the implied control cost estimates, which we develop, are scale-free.
This means that the same bound holds uniformly over a sequence of bounded domains which exhaust all of $\RR^d$.

The control cost estimates which we present are optimal in several asymptotic regimes.
More precisely, the estimate becomes optimal for the large time $T\to \infty$ and small time $T\to 0$ limit,
as well as for the homogenization limit.
The latter corresponds to a sequence of observability sets in $\RR^d$
which have a common positive density but get evenly distributed on finer and finer scales.
Effectively this leads to a control problem with control set equal to the whole domain but with a weight factor.

Last but not least, we point out two fields of analysis
where related or complementary results to spectral inequalities in control theory have been developed.
One of them is the theory of random Schr\"odinger operators.
There, uncertainty principles play a crucial role for the study of the integrated density of states and
proofs of Anderson localization.
The other is the use of uncertainty principles developed with the help
of complex or harmonic analysis to study semi-norms on $L^p$-spaces.

\section{Scale-free spectral inequalities based on complex analysis}

\label{sec:UCP_Logvinenko-Sereda}
In this section we give an overview of scale-free spectral inequalities obtained through complex analytical methods, in contrast to the ones obtained through Carleman
estimates, discussed in a subsequent section. The term \emph{scale-free} stands for the independence of the estimates on
the size of the underlying domain.
In particular, only a dependence on the dimension, on the geometry of the observability set, and on the class of functions considered is present.

These inequalities deal with the class of $L^p$-functions on $\RR^d$ with compactly supported Fourier transform or with $L^p$-functions on the $d$-dimensional torus
with sides of length $2\pi L$, $L>0$, with active Fourier frequencies contained in a parallelepiped of $\RR^d$, and with observability sets
which are measurable and well-distributed in $\RR^d$ in the following sense:

\begin{definition}
Let $S$ be a subset of $\RR^d$, $d\in\NN$. We say that $S$ is a \emph{thick set} if it is measurable  and there exist
$\gamma\in(0,1]$ and $a=(a_1,\ldots,a_d)\in \RR^d_+$
such that
\begin{equation*}
\abs{S\cap (x+[0,a_1]\times\ldots\times [0,a_d])}\geq\gamma\prod_{j=1}^d a_j,\quad \forall \ x \in\RR^d.
\end{equation*}
Here $\abs{\cdot}$ denotes the Lebesgue measure in $\RR^d$. We will call thick sets $(\gamma,a)$-thick to emphasise the parameters.
\end{definition}

This geometric condition relates the volume of cubes to the volume of the part of these cubes inside $S$. It can equivalently also be
formulated with respect to balls in $\RR^d$, in which case $a\in\RR^d_+$ is replaced by a radius $r>0$. The latter is considered in the
proof of Lemma \ref{lem:thick} below.

Before presenting the most current results, we discuss how these spectral inequalities and the above geometric condition were identified originally.

\subsection{Earlier literature and historical development: Equivalent norms on subspaces}
Let $d\in\NN$, $p\in [1,\infty]$, $\Omega\subset\RR^d$, and $S\subset\RR^d$ be  measurable subsets.
We define
\[
F(\Omega, p):=\{f\in L^p(\RR^d) \colon \supp\hat{f}\subset \Omega\},
\]
where $\hat{f}$ is the Fourier transform of $f$. If $\Omega$ is bounded, we ask for which sets $S$
there exists a constant $C=C(S,\Omega)>0$ such that
\begin{equation}\label{eq:uncertainty}
 \norm{f}_{L^p(S)}\geq C \norm{f}_{L^p(\RR^d)}, \qquad \forall \ f\in F(\Omega, p).
\end{equation}

Since $\norm{\cdot}_{L^p(S)}$ defines a semi-norm on $F(\Omega, p)$ and $\norm{\cdot}_{L^p(S)}\leq \norm{\cdot}_{L^p(\RR^d)}$,
we are actually asking for which sets $S$ this semi-norm defines a norm equivalent to the $L^p$-norm on $\RR^d$.

This question was (at the best of our knowledge) first considered by Panejah in \cite{Panejah-61}.
The author treated the case $p=2$
and characterized the class of sets $S$ satisfying \eqref{eq:uncertainty} through a property of their complement.
Indeed, our initial question is equivalent to the problem for which sets $S$ there exists a constant $\tilde{C}=\tilde{C}(S^c,\Omega)\in (0,1)$,
$S^c$ being the complement of $S$ in $\RR^d$, such that
\begin{equation}\label{eq:uncertainty-complement}
 \norm{f}_{L^p(S^c)}\leq \tilde{C} \norm{f}_{L^p(\RR^d)}, \qquad \forall \ f\in F(\Omega, p).
\end{equation}

If we set
\begin{equation*}
\rho(S^c,p):=\sup\{ \norm{f}_{L^p(S^c)} \colon f\in F(\Omega,p), \norm{f}_{L^p(\RR^d)}=1\},
\end{equation*}
then \eqref{eq:uncertainty-complement} is satisfied for a $\tilde{C}<1$ if and only if $\rho(S^c,p)< 1$. The main result in \cite{Panejah-61} is a necessary condition
for $\rho(S^c,2)<1$.

\begin{theorem}[\cite{Panejah-61}]\label{thm:panejah61}
 Let $d\in\NN$. Let $S\subset\RR^d$ be a measurable set and $S^c$ its complement in $\RR^d$. Let $B(x,r)$ be the ball in $\RR^d$ centered at $x$ of radius $r>0$. If
\begin{equation}  \label{eq:not-thick}
 \beta(S^c):=\lim_{r\to +\infty} \sup_{x\in\RR^d} \frac{\abs{S^c\cap B(x,r)}}{\abs{B(x,r)}} =1,
\end{equation}
  then $\rho(S^c,2)=1$.
\end{theorem}

Let us observe that Eq.~\eqref{eq:not-thick} in Theorem~\ref{thm:panejah61} is just a different characterization for $S$ not being a thick set.
Indeed, we have the following:

\begin{lemma}\label{lem:thick}
Let $d\in\NN$. Let $S\subset\RR^d$ be a measurable set with complement $S^c$.
Then $S$ is thick if and only if $ \beta(S^c)<1$.
\end{lemma}

\begin{proof}
Let us assume that $S$ is not thick. Then for all $\gamma, r>0$ there exists a ball $B(x_{\gamma, r},r)$
centered at some point $x_{\gamma, r}\in\RR^d$ dependent on $\gamma$ and $r$, such that $\abs{S\cap B(x_{\gamma,r},r)}<\gamma \abs{B(x_{\gamma, r},r)}$.
Let now $r>0$ and choose $\gamma=1/r$, then
\begin{equation*}
\inf_{x\in\RR^d}\frac{\abs{S\cap B(x,r)}}{\abs{B(x,r)}}\leq \frac{\abs{S\cap B(x_{1/r,r},r)}}{\abs{B(x_{1/r,r},r)}} < \frac{1}{r},
\end{equation*}
which implies
\begin{equation*}
\lim_{r\to +\infty}\inf_{x\in\RR^d}\frac{\abs{S\cap B(x,r)}}{\abs{B(x, r)}} =0.
\end{equation*}
Since $\abs{S^c\cap B(x,r)}=\abs{B(x,r)}-\abs{S\cap B(x,r)}$, we obtain
\begin{equation*}
\lim_{r\to+\infty}\sup_{x\in\RR^d}\frac{\abs{S^c\cap B(x,r)}}{\abs{B(x,r)}} = 1 - \lim_{r\to+\infty}\inf_{x\in\RR^d} \frac{\abs{S\cap B(x,r)}}{\abs{B(x,r)}}=1,
\end{equation*}
that is, $\beta(S^c)=1$.

Conversely, if $S$ is a thick set, we find some positive $\gamma$ and $r$ such that
\begin{equation*}
\inf_{x\in\RR^d}\frac{\abs{S\cap B(x,r)}}{\abs{B(x,r)}}   \geq \gamma
\end{equation*}
and hence
$\lim_{r\to +\infty}\inf_{x\in\RR^d}\frac{\abs{S\cap B(x,r)}}{\abs{B(x, r)}} \geq \gamma$.
Arguing as above, we see that  $\beta(S^c)\leq 1-\gamma <1$, which completes the proof.
\end{proof}
To summarize, we have collected the following implications (at least for $p=2$):
\begin{multline*}
\eqref{eq:uncertainty} \ \text{holds for some} \ C(S,\Omega)>0
\Leftrightarrow
\eqref{eq:uncertainty-complement}\ \text{holds for some} \ \tilde C(S^c,\Omega)\in (0,1)
\\
\Leftrightarrow
\rho(S^c, 2) <1
\Rightarrow
\beta(S^c) <1
\Leftrightarrow
S \ \text{is thick}
\end{multline*}

So, this leaves open the (hard) question whether thickness of $S$ is a sufficient criterion to ensure the equivalence of norms in \eqref{eq:uncertainty}.

In the subsequent paper \cite{Panejah-62} Panejah shows that in dimension one the condition $\beta(S^c)<1$ is also sufficient for $\rho(S^c,2)<1$,
while in higher dimensions he provides a sufficient condition unrelated to the necessary one.
In both papers, the methods used rely essentially on $L^2$-properties of the Fourier transform.

A different approach was taken by Logvinenko \& Sereda \cite{Logvinenko-Sereda-74} and Kacnel'son \cite{Katsnelson-73}.
Using the theory of harmonic functions they considered the case $p\in (0,\infty)$ and, almost simultaneously, proved the following theorem.

\begin{theorem}[\cite{Logvinenko-Sereda-74, Katsnelson-73}]\label{thm:LS-Kat}
 Let $d\in\NN$, $\sigma>0$, $p\in [1,\infty)$, and $S\subset\RR^d$ be a measurable set.
Then the following statements are equivalent:
\begin{itemize}
 \item [(i)] $S$ is a thick set;
 \item [(ii)] there exists a constant $C=C(S,\sigma)>0$ such that for all
entire functions $f:\CC^d\rightarrow \CC$
satisfying $f|_{\RR^d}\in L^p(\RR^d)$
and
\[
\limsup_{\abs{z_1}+\ldots+\abs{z_d}\to\infty}\left(\sum_{i=1}^d\abs{z_i}\right)^{-1}\ln f(z) \leq \sigma
\]
we have
\begin{equation}\label{eq:LS-Kat}
 \norm{f}_{L^p(S)}\geq C\norm{f}_{L^p(\RR^d)}.
\end{equation}
\end{itemize}
\end{theorem}
In addition, they exhibit the dependence on $\sigma$ establishing the relation
\[
C=c_1 e^{\sigma c_2},
\]
where $c_1$ and $c_2$ depend only on the thickness parameters of $S$ and the dimension $d$.

We observe that a function $f$ satisfying the assumption in part (ii) is called
entire $L^p$-functions of exponential type $\sigma$.
Equivalently, the space of such functions is
the space of functions with Fourier transform supported in ball of radius $\sigma$ (see for example \cite[Theorem IX.11]{ReedS-vol2-75} or \cite{Andersen-14}).
Hence, Theorem \ref{thm:LS-Kat} may be regarded as the first quantitative statement related to the problem formulated in \eqref{eq:uncertainty}.

\subsection{Current state of the art}
A quantitatively improved version of Theorem \ref{thm:LS-Kat} was given in early 2000's by Kovrijkine (see \cite{Kovrijkine-01} for the one dimensional case and
\cite{Kovrijkine-00} for the higher dimensional case). Using complex analytical techniques, he
shows that the constant $C(S,\Omega)$ in \eqref{eq:LS-Kat} depends polynomially on the thickness parameters of the set $S$. Moreover,
he analyzes the case when the support of the Fourier transform is contained in a finite union of
parallelepipeds, which may or may not be disjoint.
His approach is inspired by work of Nazarov \cite{Nazarov-93}, which studies topics related to the classical Turan Lemma \cite{Turan-46}.

More precisely, Kovrijkine proved the following statement.
\begin{theorem}[\cite{Kovrijkine-00}]   \label{thm:kovrijkine}
 Let $d\in\NN$, $p\in[1,\infty]$, and let $S$ be a $(\gamma, a)$-thick set in $\RR^d$.
\begin{itemize}
 \item [(i)] Let $J$ be a parallelepiped with sides of length $b_1,\ldots, b_d$ parallel to the coordinate axes and let $f\in F(J,p)$.
Set $b=(b_1,\ldots, b_d)$, then
\begin{equation} \label{eq:kovrijkine-1}
 \norm{f}_{L^p(S)}\geq \left(\frac{\gamma}{K_1^d}\right)^{K_1(a\cdot b+d)} \norm{f}_{L^p(\RR^d)},
\end{equation}
where $K_1>0$ is a universal constant.
\item [(ii)] Let $n\in\NN$ and let $J_1, \ldots,J_n$ be parallelepipeds with sides parallel to the coordinate axes
and of length $b_1,\ldots,b_d$. Let $f\in F(J_1\cup\ldots\cup J_n,p)$ and set $b=(b_1,\ldots, b_d)$. Then
\begin{equation}\label{eq:kovrijkine-2}
\norm{f}_{L^p(S)}\geq\left(\frac{\gamma}{K_2^d}\right)^{\left(\frac{K_2^d}{\gamma}\right)^n a\cdot b +n -\frac{p-1}{p}}\norm{f}_{L^p(\RR^d)},
\end{equation}
for $K_2>0$ a universal constant.
\end{itemize}
Here $a\cdot b$ denotes the Euclidean inner product in $\RR^d$.
\end{theorem}

The different nature of the constants in \eqref{eq:kovrijkine-1} and \eqref{eq:kovrijkine-2} originates in the different approaches used in the proofs.
While the bound in \eqref{eq:kovrijkine-2} allows for more general situations, it is substantially weaker than \eqref{eq:kovrijkine-1} in the
case $n=1$. The bound in \eqref{eq:kovrijkine-1}, however, is essentially optimal, which is exhibited in the following example
(see also \cite{Kovrijkine-00}).

\begin{example}\label{example:optimality-Kov-full-space}
Let $d\in\NN$, $p\geq 1$, $a_1=\ldots = a_d=1$, and $\gamma\in (0,1)$. We choose $b>0$ such that $\NN\ni\alpha:= b/(4\pi)\geq 3$. We consider the $1$-periodic set $A$ in $\RR$ such that
$A\cap \left[-\frac{1}{2},\frac{1}{2}\right]=\left[-\frac{1}{2}, -\frac{1}{2}+\frac{\gamma}{2}\right]\cup \left[ \frac{1}{2} -\frac{\gamma}{2}, \frac{1}{2}\right]$,
and define the set $S:=A\times\RR^{d-1}$.
Clearly, $S$ is a $(\gamma, 1)$-thick set in $\RR^d$. Let now $g\colon\RR^{d-1}\rightarrow \CC$ be an $L^p$-function such that $\supp\hat{g}\subset B(0,r)\subset\RR^{d-1}$
for some $r< b/4$, and let $f\colon\RR\rightarrow\RR$ defined as $f(x_1):= \left(\frac{\sin(2\pi x_1)}{x_1}\right)^\alpha$.
Since $\supp\hat{f}\subset \left[-\frac{b}{2}, \frac{b}{2}\right]$, the function
\[
\varphi\colon\RR^d\rightarrow \CC, \quad \varphi(x)=f(x_1)g(x_2,\ldots,x_d)
\]
has Fourier Transform supported in a cylinder inside the cube $\left[-\frac{b}{2}, \frac{b}{2}\right]^{d}$.
Theorem \ref{thm:kovrijkine}(i) says
\[
 \norm{\varphi}_{L^p(S)}\geq \left(\frac{\gamma}{K_1^d}\right)^{K_1(d b+d)} \norm{\varphi}_{L^p(\RR^d)},
\]
for a constant $K_1>0$.
We now show that the $L^p$-norm of $\varphi$ on $S$ can also be bounded from above by a constant of type $\gamma^b$. In order to do so,
it is enough to bound the $L^p$-norm of $f$ on $A$ from above.

We first observe that $\norm{f}_{L^p(\RR)}\geq 1$.
Then, taking into account that $\sin(2\pi t)/t\le 6\pi(1/2-t)$ for all $t\in[0,1/2]$, we calculate
\begin{align*}
\frac{\norm{f}_{L^p(A)}}{ \norm{f}_{L^p(\RR)}} &
\leq \left(\int_A \Big\vert\frac{\sin(2\pi x_1)}{x_1}\Big\vert^{p\alpha -2}\Big\vert\frac{\sin(2\pi x_1)}{x_1}\Big\vert^2 \drm x_1\right)^{1/p}\\
& \leq \left( \sup_{x_1\in A} \Big\vert \frac{\sin(2\pi x_1)}{x_1} \Big\vert^{p\alpha-2} \int_A\Big\vert\frac{\sin(2\pi x_1)}{x_1}\Big\vert^2 \drm x_1 \right)^{1/p}\\
& \leq \left( \sup_{x_1\in A} \Big\vert \frac{\sin(2\pi x_1)}{x_1} \Big\vert \right)^{\alpha-2/p}(2\pi^2)^{1/p}\\
& = \left( \frac{\sin(2\pi(1/2-\gamma/2))}{1/2-\gamma/2}\right)^{\alpha-2/p}(2\pi^2)^{1/p}\\
& \leq (2\pi^2)^{\alpha-2/p}\left(\frac{\gamma}{2}6\pi\right)^{\alpha-2/p} = \left(\frac{\gamma}{1/(6\pi^3)}\right)^{\alpha-2/p}.
\end{align*}
Using $\alpha -2/p=\frac{b}{4\pi} -2/p \geq \frac{b}{4\pi}-2 \geq 1$, we obtain for $\gamma < 1/(6\pi^3)$
\[
 \norm{f}_{L^p(A)} \leq \left(\frac{\gamma}{1/(6\pi^3)}\right)^{\frac{b}{4\pi} -2}\norm{f}_{L^p(\RR)}.
\]
Hence, by separation of variables we conclude
\begin{align*}
 \norm{\varphi}_{L^p(S)} & =\norm{f}_{L^p(A)}\norm{g}_{L^p(\RR^{d-1})} \\
& \leq \left(\frac{\gamma}{1/(6\pi^3)}\right)^{\frac{b}{4\pi} -2}\norm{f}_{L^p(\RR)}\norm{g}_{L^p(\RR^{d-1})}\\
& = \left(\frac{\gamma}{1/(6\pi^3)}\right)^{\frac{b}{4\pi} -2}\norm{\varphi}_{L^p(\RR^d)},
\end{align*}
which shows the optimality of the $\gamma^b$ term.
\end{example}

For $p=2$, the statement of Theorem \ref{thm:kovrijkine}(i) can be easily turned into a spectral inequality.
Let $E> 0$ and let $\chi_{(-\infty, E]}(-\Delta_{\RR^d})$ be the spectral projector of $-\Delta_{\RR^d}$ up to energy $E$, $\Delta_{\RR^d}$ being the Laplacian on $\RR^d$.
Then
\[
 \chi_{(-\infty, E]}(-\Delta_{\RR^d})\colon L^2(\RR^d)\longrightarrow \{f\in L^2(\RR^d) \ \colon \ \supp\hat{f}\subset B(0,\sqrt{E})\},
\]
where $B(0,\sqrt{E})$ is the Euclidean ball with center 0 and radius $\sqrt{E}$.
Clearly,
\[
\ran (\chi_{(-\infty, E]}(-\Delta_{\RR^d}))\subset \{f\in L^2(\RR^d) \ \colon \ \supp\hat{f}\subset [-\sqrt{E}, \sqrt{E}]^d\}.
\]
Therefore, as explained in \cite[\S 5]{EgidiV-18}, Theorem \ref{thm:kovrijkine}(i) implies:

\begin{corollary}\label{cor:spectral-inequality-full-space}
 Let $d\in\NN$. There exists a constant $K_1>0$ such that for all $E>0$, all $(\gamma, a)$-thick sets $S$, and all $f\in\ran (\chi_{(-\infty, E]}(-\Delta_{\RR^d}))$
we have
\begin{equation*} 
 \norm{f}_{L^2(S)}\geq \left(\frac{\gamma}{K_1^d}\right)^{K_1(2\sqrt{E}\norm{a}_{1}+d)} \norm{f}_{L^2(\RR^d)},
\end{equation*}
where $\norm{a}_{1}=a_1+\ldots+a_d$.
\end{corollary}

Using similar techniques as in \cite{Kovrijkine-00},
Logvinenko-Sereda-type estimates have been recently established also on the torus $\TT^d_L=[0,2\pi L]^d$ with sides of length $2\pi L$, $L>0$, $d\in\NN$,
for $L^p(\TT^d_L)$-functions with active Fourier frequencies contained in a parallelepiped of arbitrary size, see \cite{EgidiV-16-arxiv}. This leads to a
spectral inequality for linear combinations of eigenfunctions of the Laplacian on $\TT^d_L$ with suitable boundary conditions, see \cite[\S 5]{EgidiV-18}.

For $f\in L^p(\TT^d_L)$ we adopt the convention:
\begin{equation*}
\hat{f} \colon \left(\frac{1}{L}\ZZ\right)^d\rightarrow\RR^d,\qquad
\hat{f}\left(\frac{k_1}{L},\ldots,\frac{k_d}{L}\right)=\frac{1}{(2\pi L)^d}\int_{\TT_L^d} f(x)e^{-i\frac{1}{L}x\cdot k}\drm x.
\end{equation*}
In particular, $\supp\hat{f}\subset \left(\frac{1}{L}\ZZ\right)^d\subset\RR^d$.

\begin{theorem}[\cite{EgidiV-16-arxiv}]\label{thm:log-ser-cubes}
Let $p\in[1,\infty]$ and $L>0$. Let $\TT^d_L=[0,2\pi L]^d$, $f\in L^p(\TT^d_L)$,
and $S$ be a  $(\gamma, a)$-thick set with $a=(a_1,\ldots,a_d)$ such that $0<a_j\leq 2\pi L$ for all $j=1,\ldots,d$.
\begin{itemize}
 \item [(i)] Assume that $\supp\hat{f}\subset J$, where $J$ is a parallelepiped in $\RR^d$ with sides of length $b_1,\ldots, b_d$ and parallel to coordinate axes.
Set $b=(b_1,\ldots, b_d)$, then
\begin{equation}\label{eq:log-ser-cubes-1}
\norm{f}_{L^p(S\cap \TT^d_L)}\geq\Big(\frac{\gamma}{K_3^d}\Big)^{K_3 a\cdot b+\frac{6d+1}{p}}\norm{f}_{L^p(\TT^d_L)},
\end{equation}
where $K_3>0$ is a universal constant.
 \item [(ii)] Let $n\in\NN$ and assume that $\supp\widehat{f}\subset \bigcup_{l=1}^n J_l$, where each $J_l$ is a parallelepiped in $\RR^d$
with sides of length $b_1,\ldots,b_d$ and parallel to coordinate axes. Set $b=(b_1,\ldots, b_d)$, then
\begin{equation*}
 \norm{f}_{L^p(S\cap\TT^d_L)}\geq\Big(\frac{\gamma}{K_4^d}\Big)^{\big(\frac{K_4^d}{\gamma}\big)^n a\cdot b+n-\frac{(p-1)}{p}}\norm{f}_{L^p(\TT^d_L)},
\end{equation*}
for $K_4>0$ a universal constant.
\end{itemize}
Here $a\cdot b$ denotes the Euclidean inner product in $\RR^d$.
\end{theorem}

We emphasize that these estimates are uniform for all $L\geq (2\pi)^{-1}\max_{j=1,\ldots, d} a_j$
and are independent of the position of the  parallelepipeds $J_l$.
Note that for growing $L$ the number of possible Fourier frequencies in the set $\bigcup_{l=1}^n J_l$ grows unboundedly.

Let us also note that in~\cite[Corollary 3.3]{TenenbaumT-11}, related techniques from complex analysis, in particular a version of the Turan Lemma, are
used to establish an estimate similar to the one in Theorem~\ref{thm:log-ser-cubes}. However, there the control set $S$ is assumed to contain a parallelepiped
and the constant comparing $\|\cdot\|_{L^2(S)}$ and $\|\cdot\|_{L^2(\TT^d_L)}$ depends on its volume.

Comparing (i) and (ii) of the above theorem, we again see that, although (ii) allows for more general situations, the corresponding constant is worse that
the one in (i) in the case $n=1$. Example \ref{example:optimal-EV-cubes} below, inspired by Example \ref{example:optimality-Kov-full-space},
shows that for general $L^p$-functions on $\TT^d_L$ estimate \eqref{eq:log-ser-cubes-1} is optimal
up to the unspecified constant $K_3$. However, this bound may be improved once special classes of functions are considered,
for example Fourier series with few, but spread out Fourier coefficients, as discussed in Example \ref{example:not-optimal-EV-cubes}.
\begin{example}\label{example:optimal-EV-cubes}
Let $a_1=\ldots = a_d=1$,
$p\geq 1$, $b\geq 8\pi$, and $\varepsilon\in (0,1)$. We consider the set
\[
S=A_1\times\ldots\times A_d\subset \RR^d
\]
such that each $A_j$ is 1-periodic and $A_j\cap[0,1]= \left[\frac{1}{2}-\frac{\varepsilon}{2},\frac{1}{2}+\frac{\varepsilon}{2} \right]$.
Then, $S$ is $(\gamma, 1)$-thick in $\RR^d$ with $\gamma=\varepsilon^d$.

Let now $\NN\ni\alpha:=\lfloor\frac{b}{4\pi}\rfloor$ and $L=1/(2\pi)$.
On the torus $\TT^1_L=[0,2\pi L]=[0,1]$ and on its $d$-dimensional counterpart $\TT^d_L=[0,1]^d$ we consider the functions
\begin{align*}
& f\colon [0,1]\rightarrow \RR, \quad f(x):=(\sin(2\pi x))^\alpha\\
& g\colon [0,1]^d \rightarrow \RR, \quad g(x):=\prod_{j=1}^d f(x_j)= \prod_{j=1}^d \sin(2\pi x_j)^\alpha.
\end{align*}
Clearly, $\supp\hat{f}\subset [-2\pi\alpha, 2\pi\alpha]\subset \left[-\frac{b}{2},\frac{b}{2}\right]$ and $\supp \hat{g}\subset\left[-\frac{b}{2},\frac{b}{2}\right]^d$,
and the Fourier coefficients are uniformly spaced.

Consequently, by Theorem \ref{thm:log-ser-cubes}(i) we know
\[
 \norm{g}_{L^p(S\cap [0,1]^d)}\geq\Big(\frac{\varepsilon^d}{K_3^d}\Big)^{K_3 d b+\frac{6d+1}{p}}\norm{g}_{L^p([0,1]^d)}.
\]
We now show that the prefactor cannot be improved qualitatively.
To obtain an upper bound on $\norm{g}_{L^p(S\cap [0,1]^d)}$ we proceed as follows.
By separation of variables, $\norm{g}_{L^p(S\cap[0,1]^d)} \allowbreak= \prod_{j=1}^d \norm{f}_{L^p(A_j\cap [0,1])}$ and similarly for $\norm{g}_{L^p([0,1]^d)}$.
It is therefore enough to analyze the $L^p$-norm of $f$ on $A_1\cap [0,1]$.

By Jensen's inequality we have
\[
 \norm{f}_{L^p([0,1])}^p=\int_0^1\abs{\sin(2\pi x)}^{p\alpha} \drm x \geq \left(\int_0^1 \abs{\sin(2\pi x)}\drm x\right)^{p\alpha}= \left(\frac{2}{\pi}\right)^{p\alpha}.
\]
By symmetry of the sinus function, $\sin x \leq x$, the choice of $\alpha$, and the change of variable $y=2\pi x$, we estimate
\begin{align*}
\frac{ \norm{f}_{L^p(A_1\cap [0,1])}}{ \norm{f}_{L^p([0,1])}}
&
\leq \left(\frac{\pi}{2}\right)^{\alpha}\left(\int_{A_1\cap [0,1]} \abs{\sin(2\pi x)}^{p\alpha}\drm x\right)^{1/p}
\\
& = \left(\frac{\pi}{2}\right)^{\alpha}
\left(\frac{1}{\pi}\int_{0}^{\pi \varepsilon} \sin^{p\alpha}(y)\drm y\right)^{1/p}
\\
& \leq \left(\frac{\pi}{2}\right)^{\alpha}
\left(\frac{1}{\pi}\int_{0}^{\pi \varepsilon} y^{p\alpha}\drm y\right)^{1/p}
\\
& = \left(\frac{\pi}{2}\right)^{\alpha}
\left(\frac{1}{\pi} \frac{(\pi\varepsilon)^{1+p\alpha}}{1+p\alpha}\right)^{1/p}
\\
&
\leq \left(\frac{\varepsilon}{(2/\pi^2)}\right)^{\alpha+1/p}.
\end{align*}

Using $\alpha +1/p=\lfloor\frac{b}{4\pi}\rfloor +1/p \geq \frac{b}{4\pi} -1 \geq 1$
we obtain for $\varepsilon < 2/\pi^2$
\begin{equation*}
 \norm{f}_{L^p(A_1\cap [0,1])} \leq \left(\frac{\varepsilon}{(2/\pi^2)}\right)^{\frac{b}{4\pi} -1}\norm{f}_{L^p([0,1])},
\end{equation*}
which holds also for $\varepsilon \geq  2/\pi^2$ trivially.
Consequently,
\begin{equation*}
 \norm{g}_{L^p(S\cap[0,1]^d)}
 \leq
 \left(\frac{\gamma}{(2/\pi^2)^d}\right)^{\frac{b}{4\pi} -1}\norm{g}_{L^p([0,1]^d)}.
\end{equation*}
This shows that in general we cannot obtain a Logvinenko-Sereda
constant which is qualitatively better than $(\gamma/ c^d)^{c(b+d)}$, for some $c>0$.
\end{example}

\begin{example}\label{example:not-optimal-EV-cubes}
Let $b\in \NN$, $\gamma\in (0,1)$, $S$ be the $1$-periodic set such that $S\cap [0,1]=[0,\gamma]$, and
$f\colon [0,1]\to \RR$ be defined as $f(x):=\sin(2b\pi x)$. This function has two non-zero Fourier coefficients at $-2b\pi$ and $2b\pi$,
growing apart as $b$ increases. For the $L^1$-norm of $f$ on $[0,1]$ and $[0,\gamma]$ we calculate
\begin{equation*}
 \frac{\norm{f}_{L^1([0,\gamma])}}{\norm{f}_{L^1([0,1])}} \leq \frac{\pi}{2}\int_0^\gamma 2b\pi x \; \drm x = \frac{\pi^2}{2} b\gamma^2,
\end{equation*}
suggesting a behavior of type $b\gamma^2$ instead of $\gamma^b$ as in Theorem \ref{thm:log-ser-cubes}(i).
\end{example}

As anticipated, the case $p=2$ in Theorem \ref{thm:log-ser-cubes}(i) is of particular interest, since it can be interpreted as a statement for functions in the range
of the spectral projector of $-\Delta_{\TT^d_L}$ with periodic, Dirichlet, or Neumann boundary conditions.
Let $\Delta_{\TT^d_L}^P, \Delta_{\TT^d_L}^D, \Delta_{\TT^d_L}^N$ be the Laplacian on $\TT^d_L$ with periodic, Dirichlet, and Neumann boundary
conditions, respectively. To shorten the notation we set $\bullet\in\{P,D,N\}$. Let $\chi_{(-\infty, E]}(-\Delta_{\TT^d_L}^\bullet)$ be the spectral projector of $-\Delta_{\TT^d_L}^\bullet$
up to energy $E>0$. Namely, let $\lambda^\bullet$ and $\phi^\bullet_{\lambda^\bullet}$ be the eigenvalues and corresponding eigenfunctions of $-\Delta_{\TT^d_L}^\bullet$, then
\begin{equation*}
\chi_{(-\infty, E]}(-\Delta_{\TT^d_L}^\bullet): L^p(\TT^d_L)\longrightarrow \left\{ \sum_{\lambda^\bullet\leq E}\alpha_{\lambda^\bullet}\phi^\bullet_{\lambda^\bullet}(x) \ \vert \ \alpha_{\lambda^\bullet}\in\CC \right\}.
\end{equation*}
Similarly as before, Theorem \ref{thm:log-ser-cubes}(i) implies by simple arguments performed in \cite[\S 5]{EgidiV-18}:

\begin{corollary}\label{cor:cubes-spectral}
 Let $d\in\NN$, and let $\TT^d_L=[0,2\pi L]^d$, $L>0$. There exists a universal constant $K_5>0$ such that for all $L>0$,
all $(\gamma,a)$-thick sets $S\subset \RR^d$ with $a=(a_1,\ldots, a_d)$ such that
$0< a_j\leq 2\pi L$ for all $j=1,\ldots,d$, all $E>0$, and all $f\in \ran\big(\chi_{(-\infty, E]}(-\Delta_{\TT^d_L}^\bullet)\big)$ we have
\begin{equation}\label{eq:spectral-cube}
\norm{f}_{L^2(S\cap\TT^d_L)}\geq\Big(\frac{\gamma}{K_5^d}\Big)^{K_5 \sqrt{E}\norm{a}_{1}+\frac{6d+1}{2}}\norm{f}_{L^2(\TT^d_L)},
\end{equation}
where $\norm{a}_{1}=a_1+\ldots+a_d$.
\end{corollary}

In the case of periodic boundary conditions, Corollary~\ref{cor:cubes-spectral} is a direct consequence of Theorem \ref{thm:log-ser-cubes}(i):
Since the eigenfunctions of $-\Delta_{\TT^d_L}^P$ are $e^{i(k/L)\cdot x}$ (up to a normalization factor), corresponding to eigenvalues $\norm{k}_{2}^2/L^2$, $k\in\ZZ^d$,
the Fourier frequencies of any $f\in \ran\big(\chi_{(-\infty, E]}(-\Delta_{\TT^d_L}^P)\big)$ are contained in $[-\sqrt{E}, \sqrt{E}]^d$,
and the statement follows immediately.

In contrast, when Dirichlet or Neumann boundary conditions are considered, the respective eigenfunctions do not have Fourier frequencies
contained in a compact set. However, once these functions are extended to functions on $\TT^d_{2L}$ in a suitable way depending on the
boundary conditions, the Fourier frequencies of the extensions are concentrated in $[-\sqrt{E}, \sqrt{E}]^d$.
Correspondingly, one can construct a new thick set with controllable thickness parameters
by first extending $S\cap \TT^d_L$ to $\TT^d_{2L}$ using reflections with respect to the boundary of $\TT^d_L$,
and then taking the union of translates of this set with respect to the group $(4\pi L\ZZ)^d$.
Finally, Theorem \ref{thm:log-ser-cubes}(i) applied to the extensions and the new thick set yields Corollary \ref{cor:cubes-spectral}.
For more details we refer the reader to \cite[\S 5]{EgidiV-18}.

\begin{remark}\label{rem:strip-spectral}
Recently, a Logvinenko-Sereda-type estimate has also been obtained for $L^2$-functions on the infinite strip $\Omega_L:=\TT^{d-1}_L\times\RR$,
$d\geq 2$ and $L>0$, having finite Fourier series as functions on $\TT^{d-1}_L$ and compactly supported Fourier transform as functions on $\RR$.
In this case, the set $S\subset\RR^d$ is assumed to be thick with parameters $a=(a_1,\ldots, a_d)\in\RR^d_+$ such that $a_j\leq 2\pi L$ for $j\in\{1,\ldots, d-1\}$,
and $\gamma\in (0,1]$, see \cite[Theorem 9]{Egidi}. With similar arguments as in \cite[\S 5]{EgidiV-18}, we obtain, as a consequence, a corresponding
variant of Corollary \ref{cor:cubes-spectral} on the strip, that is, a spectral inequality analogous to~\eqref{eq:spectral-cube} for functions
in the range of the spectral projector $\chi_{(-\infty, E]}(-\Delta_{\Omega_L})$, where $-\Delta_{\Omega_L}$ is the Laplacian on $\Omega_L$ with
either Dirichlet or Neumann boundary conditions.
\end{remark}

\section{Scale-free spectral inequalities based on Carleman estimates}
\label{sec:Carleman}

Most of the results which we present here have originated in works devoted to the
spectral theory and asymptotic analysis of evolution of solutions of random Schr\"odinger equations.
The interested reader may consult for instance the monographs \cite{Stollmann-01,Veselic-08,AizenmanW-15} for an exposition of this research area.
In this theory one is (among others) interested in lifting estimates for eigenvalues.
The particular task we want to discuss here can be formulated in
operator theoretic language in the following way: Given a self-adjoint and lower semi-bounded operator $H$ with purely discrete spectrum
$\lambda_1(H)\leq \ldots \leq \lambda_k(H)\leq \ldots $, a parameter interval $I \subset \RR$,
a cut-off energy $E \in\RR$,  and a positive semi-definite perturbation $B$, find a positive constant $C$ such that
\begin{equation*}
  \frac{\drm}{\drm t} \lambda_k(H+tB) \geq C
\end{equation*}
for all indices $k \in \NN$ for which the associated eigenvalue curve $I \ni t\mapsto \lambda_k(H+tB)$ stays
below the level $E$ for all $t \in I$.
Depending on the properties of $H$ and $B$, this exercise may be trivial, demanding, or impossible, so we should say a bit more about the structure of the operators of interest.

The self-adjoint Hamiltonian $H$ models a condensed matter system, and studying it will require investigating it on several scales,
on the one hand the macroscopic scale of the solid and on the other the microscopic scale of atoms.
Let us explain this in more detail:
If we choose a coordinate system such that the typical distance between atomic nuclei is equal to one,
the size $L$ of the macroscopic solid may be very large -- of the order of magnitude of $10^{23}$ or so.
Hence the Hamiltonian of the system $H$ will be defined on the Hilbert space $L^2(\Lambda_L)$ where $\Lambda_L = (- L/2, L/2)^d$ and $L\gg 1$.
Since often the only possibility to understand the full system is to consider first smaller sub-systems and subsequently analyze how they interact, one is also interested in intermediate scales.
Thus in the discussion which follows, the scale $L$ will always be larger than one, but will range over many orders of magnitude.
\par
The Hamiltonian $H$ will be a Schr\"odinger operator of the form $H=-\Delta+V$ in $L^2(\Lambda_L)$.
The electric potential $V$, which mainly models the force of the atomic nuclei in the solid on an electron wave packet,
will have a characteristic length scale corresponding to the typical distance between atoms (which as above we set equal to one).
This characteristic scale could manifest itself in different ways.
For instance, $V$ may be the restriction $\chi_{\Lambda_L} V_{\per}$ of a $\ZZ^d$-periodic potential $V_{\per}\colon \RR^d\to \RR$.
It could also have a structure which is not exactly periodic but incorporates some deviations from periodicity.
Furthermore, it can happen that the exact shape of $V$ is not known.
In this case, $V$ is modeled by a random field allowing for local fluctuations.
The values of the field at two different points with a distance of order one may be correlated,
but the field will exhibit a mixing behavior on large scales.
Since we study the system in $L^2(\Lambda_L)$ for many scales $L \geq 1$, the ratio between the scale $L$ of the whole system and the scale one can grow unboundedly.
In view of this challenge, for a comprehensive understanding of the system it is required to derive so-called \emph{scale-free} results, i.e. results which hold for all  scales $L \geq 1$,
or at least for an unbounded sequence of length scales.

In the light of this multi-scale structure, the problem formulated above takes now a more specific form:
We are given a bounded potential $V$ as well as the corresponding Schr\"odinger operator $H=-\Delta+V$ on $\Lambda_L$ with self-adjoint, say Dirichlet, boundary conditions, a bounded non-negative perturbation potential $W\colon \Lambda_L\to [0,\infty)$, which is the restriction to $\Lambda_L$  of a (more or less) periodic potential $W_\per\colon \RR^d\to [0,\infty)$, an interval $I \subset \RR$, a cut-off energy $E \in\RR$ and aim to find a positive constant $C$, independent of $L\geq 1$, such that
\begin{equation}\label{eq:eigenvalue-lifting}
  \frac{\drm}{\drm t} \lambda_k(H+tW) \geq C
\end{equation}
for all indices $k \in \NN$ for which the associated eigenvalue parametrization $I \ni t\mapsto \lambda_k(H+tB)$ stays below $E$ for all $t \in I$.
The real challenge of the problem is to obtain a bound $C$ which is scale-independent.
Furthermore, $C$ should only depend on some rough features of $V$ and $W$ such as their sup-norms but not on minute details of their shape.
This is required since, as explained above, the potentials $V$ and $W$ might by modeled as realizations of a random field where it is possible to control certain global properties, but not the detailed shape of the realization.

The eigenvalue lifting bound in \eqref{eq:eigenvalue-lifting} can be derived from an \emph{uncertainty relation for spectral projectors}
of the type
\begin{equation*}
\chi_{(-\infty,E]}(H)
W
\chi_{(-\infty,E]}(H)
\geq C
\chi_{(-\infty,E]}(H) ,
\end{equation*}
where the inequality is understood in quadratic form sense.
In fact, since the operators considered so far have purely discrete spectrum, this inequality can be rewritten in terms of linear combinations
of eigenfunctions, so that the conclusion  \eqref{eq:eigenvalue-lifting} is almost immediate.
If $H$ is the Dirichlet-Laplacian on a bounded domain
the above  uncertainty relation is called \emph{spectral inequality} in the literature on control theory.
It is also sometimes called a \emph{quantitative unique continuation estimate} (for spectral projectors),
because its proof uses a refined version of the proof of the classical qualitative unique continuation principle for solutions of second order elliptic operators, based on Carleman estimates.
In the particular case where, as explained above, the constant $C$ in the estimate is independent of the length scale $L\geq 1$,
the inequality is called \emph{scale-free unique continuation estimate}.
Let us present a summary of such results derived in the context of random Schr\"odinger operators.

\subsection{Development of scale-free unique continuation estimates applicable to Schr\"o\-dinger operators with random potential}
\label{sec:eigenfunction}

We will not be able to review all publications dealing with the topic, in particular older ones,
but have to be selective due to limitations of space. As the starting point we choose an important and intuitive
result of \cite{CombesHK-03} (which was fully exploited only in \cite{CombesHK-07}).

\begin{theorem}[\cite{CombesHK-03,CombesHK-07}] \label{thm:CombesHK}
Let $E\in \RR$, $V_{\per}\colon \RR^d\to \RR$ be a measurable, bounded and $\ZZ^d$-periodic potential
and $H_L^{\per}$ the restriction of $-\Delta +V_{\per}$ to the cube $\Lambda_L$ with periodic boundary conditions.
Denote by
$\chi_{(-\infty,E]}(H_L^{\per})$ the spectral projector of $H_L^{\per}$ associated to the energy interval $(\infty, E]$.
If  $\cO \neq \emptyset$ is an open $\ZZ^d$-periodic subset of $\RR^d$ and
$W\colon \RR^d \to [0, \infty)$ a measurable, bounded and $\ZZ^d$-periodic potential such that
$W \geq \chi_\cO$ then there exists a constant $C>0$ depending on $E$, $V_{\per}$ and $W$, but not on $L\in \NN$,
such that
\begin{equation*}
\chi_{(-\infty,E]}(H_L^{\per})
\ W \chi_{\Lambda_L} \
\chi_{(-\infty,E]}(H_L^{\per})
\geq C
\chi_{(-\infty,E]}(H_L^{\per}) \quad \text{for all scales} \ L\in \NN.
\end{equation*}
\end{theorem}
The theorem gives no estimate on the constant $C>0$, since its proof  invokes a compactness argument.
Moreover, it is based on the Floquet-Bloch decomposition
and thus cannot be extended to a situation without periodicity. This explains also the restriction to integer valued scales $L\in \NN$.

An improvement of the above theorem with an explicit lower bound on $C$
was given in \cite{GerminetK-13b}.
The method which allowed to derive this quantitative estimate was  a \emph{Carleman estimate}.
It was the seminal paper \cite{BourgainK-05} which introduced Carleman estimates
to the realm of random Schr\"odinger operators and stimulated the further development.
With this tool at hand it was possible to circumvent the use of
Floquet-Bloch theory in the proof of scale-free unique continuation estimates.
Consequently, it was possible to remove the periodicity assumptions on the potential function $V$ and the set $\cO$.
They can be be replaced by a geometric condition which we define next.
\begin{definition}
Given $G,\delta > 0$, we say that a sequence $Z = (z_j)_{j \in (G \ZZ)^d} \subset \RR^d$ is \emph{$(G,\delta)$-equidistributed}, if
 \[
  \forall j \in (G \ZZ)^d \colon \quad  B(\x_j , \delta) \subset \Lambda_G + j .
\]
Corresponding to a $(G,\delta)$-equidistributed sequence $Z$ we define the set
\[
\equiset = \bigcup_{j \in (G \ZZ)^d } B(\x_j , \delta) ,
\]
see Fig.~\ref{fig:equidistributed} for an illustration. Note that the set $\equiset$ depends on $G$ and the choice of the $(G,\delta)$-equidistributed sequence $Z$.
\end{definition}
\begin{figure}[ht]\centering
\begin{tikzpicture}
\pgfmathsetseed{{\number\pdfrandomseed}}
\draw (-.2, -.2) grid (5.2,5.2);
\foreach \x in {0.5,1.5,...,4.5}{
  \foreach \y in {0.5,1.5,...,4.5}{
    \filldraw[fill=gray!70] (\x+rand*0.35,\y+rand*0.35) circle (0.15cm);
  }
}

\begin{scope}[xshift=-6cm]
\draw[] (-.2, -.2) grid (5.2,5.2);
\foreach \x in {0.5,1.5,...,4.5}{
  \foreach \y in {0.5,1.5,...,4.5}{
    \filldraw[fill=gray!70] (\x,\y) circle (1.5mm);
  }
}
\end{scope}
\end{tikzpicture}
\caption{Illustration of $\equiset \subset \RR^2$ for periodically (left) and non-periodically (right) arranged balls.\label{fig:equidistributed}}
\end{figure}
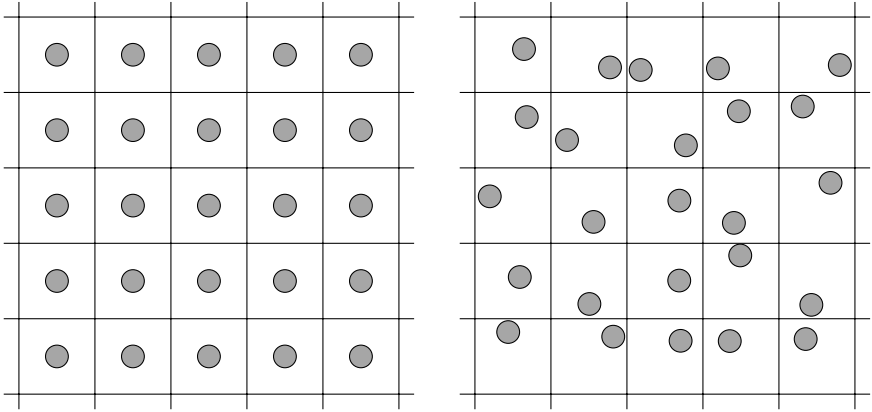
For $L > 0$ we denote by $\cD(\Delta^\per_L)$ and $\cD(\Delta^\dir_L)$ the domain of the Laplacian on $L^2(\Lambda_L)$ subject to periodic or Dirichlet boundary conditions.
With this notion at hand we formulate the following result:
\begin{theorem}[\cite{RojasMolinaV-13}]
\label{thm:RojasVeselic-UCP}
Let $E \in \RR$.
There exists a constant $K\in(0,\infty)$ depending merely on the dimension $d$,
such that for any $E \in \RR$, any $G > 0$, $\delta \in (0,G/2]$, any $(G,\delta)$-equidistributed sequence $Z$,
any measurable and bounded $V\colon {\RR^d}\to \RR$, any $L \in G \NN$ and any real-valued $\psi \in \cD(\Delta^\per_L) \cup \cD(\Delta^\dir_L)$ satisfying
$\lvert \Delta \psi \rvert \leq \lvert (V-E)\psi \rvert $ almost everywhere on $\Lambda_L$ we have
\begin{equation} \label{eq:Rojas-MolinaV-13}
\lVert \psi \rVert_{{L^2}(\Lambda_L)}
\geq
\lVert \psi \rVert_{{L^2}(\equiset \cap \Lambda_L)}
\geq
\left( \frac{\delta}{G}\right)^{K(1+G^{4/3} \lVert V-E \rVert_\infty^{2/3})}
\lVert \psi \rVert_{{L^2}(\Lambda_L)} .
\end{equation}
\end{theorem}
The last inequality implies by first order perturbation theory the lifting estimate \eqref{eq:eigenvalue-lifting}
with
\begin{equation}\label{eq:UCPconstant}
 C=\left( \frac{\delta}{G}\right)^{K(1+G^{4/3} \lVert V-E \rVert_\infty^{2/3})}.
\end{equation}
The theorem has been extended to $\RR^d$ in \cite{TautenhahnV-16}.
Lower bounds like \eqref{eq:Rojas-MolinaV-13} (with less explicit constants) have previously been known for
\begin{enumerate}[(1)]
  \item Schr\"odinger operators in one dimension (see \cite{Veselic-96}, \cite{KirschV-02b} where periodicity was assumed, and \cite{HelmV-07}, where the additional periodicity assumption was eliminated)
  \item energies $E$ sufficiently close to $\min \sigma(H)=\min \sigma(-\Delta+V)$
  (see \cite{Kirsch-96}, \cite{BourgainK-05} under a periodicity assumption and \cite{Germinet-08} without it),
   and similarly for
  \item energies $E$ sufficiently close to a spectral band edge of a periodic Schr\"odinger operator $-\Delta+V_\per$.
  (This has been implemented in \cite{KirschSS-98a} for periodic potentials using Floquet theory.)
\end{enumerate}
In the two latter cases one uses perturbative arguments, while in the one-dimensional situation one has methods from
ordinary differential equations at disposal.
The result of \cite{RojasMolinaV-13} unifies and generalizes this set of earlier results.

\begin{remark}[Dependence of the constant on parameters]
Apart form being scale independent the constant $C$ from \eqref{eq:UCPconstant}
is also explicit with respect to the model parameters.
Only the sup-norm $\lVert V \rVert_\infty$ of the potential enters, no knowledge of $V$ beyond this is used, in particular no regularity properties.
This is essential since in applications $V$ is chosen from an infinite ensemble of potentials with possibly quite different local features.
The constant is polynomial in $\delta$ and (almost) exponential in $\lVert V \rVert_\infty$.
\end{remark}

For $L > 0$ and $V \in L^\infty(\RR^d)$, we define the operator
\[
 H_L
 =
 - \Delta + V
 \quad
 \text{in $L^2(\Lambda_L)$}
\]
with Dirichlet, Neumann or periodic boundary conditions.
In view of the known scale-free uncertainty relation for periodic spectral projectors of \cite{CombesHK-03}, see Theorem \ref{thm:CombesHK}, the authors of \cite{RojasMolinaV-13} asked whether Ineq.~\eqref{eq:Rojas-MolinaV-13} holds also for linear combinations of eigenfunctions, i.e.\ for $\psi \in \ran \chi_{(-\infty , E]} (H_L)$. This is equivalent to
\begin{equation} \label{eq:full-uncertainty}
  \chi_{(-\infty,E]} (H_L) \, \chi_{\equiset\cap \Lambda_L} \, \chi_{(-\infty,E]} (H_L)
  \geq
  C \chi_{(-\infty,E]} (H_L) ,
\end{equation}
 with an explicit dependence of $C$ on the parameters $G$, $\delta$, $E$ and $\lVert V \rVert_\infty$  as in \eqref{eq:UCPconstant}.
 Here $\chi_{I} (H_L)$ denotes the spectral projector of $H_L$ associated to the interval $I$.
 If $\chi_{\equiset\cap \Lambda_L} $ is periodic and a lower bound for the potential $W$, we recover an estimate as in Theorem \ref{thm:CombesHK}.
 A partial answer was given in \cite{Klein-13}.
\begin{theorem}[\cite{Klein-13}] \label{thm:Klein-13}
There is $K = K (d)$ such that for all $E,G > 0$, all $\delta \in (0,G/2)$, all $(G,\delta)$-equidistributed sequences $Z$,
all measurable and bounded $V\colon {\RR^d}\to \RR$, all $L \in \NN$, all intervals $I \subset (-\infty , E]$ with
\[
 \vert I \rvert \leq 2 \gamma \quad \text{where} \quad
 \gamma^2 = \frac{1}{2G^4} \left(\frac{\delta}{G}\right)^{K \bigl(1+ G^{4/3}(2\lVert V\rVert_\infty + E)^{2/3} \bigr)} ,
\]
and all $\psi \in \ran \chi_{I} (H_L)$ we have
 \[
  \lVert \psi \rVert_{{L^2} (\equiset \cap \Lambda_L)}
  \geq G^4 \gamma^2 \lVert \psi \rVert_{{L^2} (\Lambda_L)} .
 \]
\end{theorem}
Again the scale-free unique continuation principle of
\cite{Klein-13} on the finite cube $\Lambda_L$
was adapted to functions on $\RR^d$ in \cite{TautenhahnV-16}.
Theorem~\ref{thm:Klein-13} left open what happens if the energy interval $I$ has length larger than $2\gamma$,
which is quite small for typical choices of $G, \delta, E, V$.
In particular, Theorem~\ref{thm:Klein-13} is not sufficient for applications in control theory which we discuss in Section~\ref{sec:uncertainty-control}.
The full answer to the above question, confirming \eqref{eq:full-uncertainty},
has been given in \cite{NakicTTV-15,NakicTTV-18}.
\begin{theorem}[\cite{NakicTTV-15,NakicTTV-18}] \label{thm:NakicTTV}
There is $K = K(d) > 0$ such that for all $G > 0$, all $\delta \in (0,G/2)$, all $(G,\delta)$-equidistributed sequences $Z$, all measurable and bounded $V: \RR^d \to \RR$, all $L \in G\NN$, all $E \geq 0$,
and all $\psi \in \mathrm{Ran} \chi_{(-\infty,E]}(H_L)$ we have
 \begin{equation*}
\lVert \psi \rVert_{{L^2} (\equiset \cap \Lambda_L)}^2
\geq C_{\sfuc} \lVert \psi \rVert_{{L^2} (\Lambda_L)}^2
\end{equation*}
where
\begin{equation*}
C_{\sfuc} = C_{\sfuc} (d, G, \delta ,  E  , \lVert V \rVert_\infty )
:=  \left(\frac{\delta}{G} \right)^{K \bigl(1 + G^{4/3} \lVert V \rVert_\infty^{2/3} + G \sqrt{E} \bigr)} .
\end{equation*}
\end{theorem}
Note that since $\Lambda_L$ is bounded, $H_L$ has compact resolvent, thus any
$\psi \in \mathrm{Ran} \chi_{(-\infty,E]} \allowbreak (H_L)$ is a finite linear combination
of eigenfunctions.
In \cite{TaeuferT-17} this assumption has been relaxed to allow certain infinite linear
combinations of eigenfunctions where the coefficients decay sufficiently fast.

\subsection{Current state of art}
Let $d \in \NN$. For $G > 0$ we say that a set $\Gamma \subset \RR^d$ is $G$-\emph{admissible}, if there exist $\alpha_i ,\beta_i \in \RR\cup \{\pm \infty\}$ with $\beta_i - \alpha_i \geq G$ for $i \in \{1,\ldots,d\}$, such that
\begin{equation}\label{eq:Gadmissible}
 \Gamma = \bigtimes_{i =1}^d (\alpha_i , \beta_i)
 \quad\text{and}\quad (-G/2 , G/2)^d \subset \Gamma.
\end{equation}
If instead of the second condition in  \eqref{eq:Gadmissible} one can only find a $\xi \in \RR^d$ such that the cube $(-G/2 , G/2)^d + \xi$ is contained in $\Gamma$, then
our assumption $(-G/2 , G/2)^d \subset \Gamma$ can be achieved by a global shift of the coordinate system.
For a $G$-admissible set $\Gamma$ and a real-valued $V \in L^\infty(\Gamma)$, we define the self-adjoint operator $H_\Gamma$ on $L^2 (\Gamma)$ as
\begin{equation*}
 H_\Gamma = -\Delta + V
\end{equation*}
with Dirichlet or Neumann boundary conditions.

\begin{theorem}[\cite{NakicTTV-18-arxiv}] \label{thm:NTTVS}
There is $K=K(d)>0$ depending only on the dimension, such that for all
$G>0$, all $G$-admissible $\Gamma \subset\RR^d$, all $\delta \in (0,G/2)$, all $(G,\delta)$-equidistributed sequences $Z$, all real-valued $V \in L^\infty (\Gamma)$, all $E \in \RR$, and all $\psi \in \ran \allowbreak \chi_{(-\infty , E]} (H_\Gamma)$ we have
\begin{equation*}
\lVert \psi \rVert_{L^2 (\equiset \cap \Gamma)}^2
\geq C_\sfuc^{(G)} \lVert \psi \rVert_{L^2 (\Gamma)}^2 ,
\end{equation*}
where
\[
 C_\sfuc^{(G)} = \sup_{\lambda \in \RR}
 \left(\frac{\delta}{G}\right)^{K \bigl(1 + G^{4/3}\lVert V-\lambda \rVert_\infty^{2/3} + G\sqrt{(E-\lambda)_+} \bigr)} ,
\]
and $t_+:=\max\{0,t\} $ for $t \in \RR$.
\end{theorem}

\begin{remark}
 If $\Gamma = \Lambda_L$ for some $L \geq G$, then $H_\Gamma$ has compact resolvent, and hence the spectrum of $H_\Gamma$ consists of a non-decreasing sequence of eigenvalues whose only accumulation point is at infinity. As a consequence, functions $\psi \in \ran \chi_{(-\infty , E]} (H_\Gamma)$ considered in Theorem~\ref{thm:NTTVS} are finite linear combinations of eigenfunctions corresponding to eigenvalues smaller than or equal to $E$. On the contrary, if $\Gamma$ is an unbounded set like $\RR^d$ or an infinite strip, the bulk of the spectrum of $H_\Gamma$ will in general consist of essential spectrum, and eigenfunctions, if any exist, might span only a subspace. Hence, the subspace $\ran \chi_{(-\infty , E]}$ might be infinite dimensional -- a challenge.
\end{remark}

The proofs of Theorems~\ref{thm:RojasVeselic-UCP} and \ref{thm:Klein-13} are heavily based on the fact that the function $\psi$ satisfies the pointwise differential inequality $\lvert \Delta \psi \rvert \leq \lvert V \psi \rvert$ almost everywhere on $\Lambda_L$, or are perturbative arguments thereof. Functions from a spectral subspace as considered in Theorem~\ref{thm:NTTVS} do in general not have this property. In what follows, we explain one main idea how to bypass this difficulty.
It is inspired by a technique developed for operators with compact resolvent in the context of control theory for the heat equation, see~e.g.\ \cite{LebeauR-95,LebeauZ-98,JerisonL-99, RousseauL-12}.

We denote by $\{ P_{H_\Gamma}(\lambda)=\chi_{(-\infty, \lambda]}(H_\Gamma) \colon \lambda \in \RR \}$ the resolution of identity of $H_\Gamma$,
and define the family of self-adjoint operators $(\mathcal{F}_t )_{t \in \RR}$ on $L^2 (\Gamma)$ by
\[
 \mathcal{F}_t
 =
 \int_{- \infty}^\infty s_t(\lambda) \mathrm{d} P_{H_\Gamma}(\lambda)
 \quad
 \text{where}
 \quad
 s_t(\lambda)=\begin{cases}
        \sinh(\sqrt{\lambda} t)/\sqrt{\lambda}, & \lambda > 0 ,\\
        t, & \lambda=0,\\
        \sin(\sqrt{- \lambda} t)/\sqrt{- \lambda}, & \lambda <0.
\end{cases}
\]
The operators $\mathcal{F}_t$ are self-adjoint, lower semi-bounded, and satisfy $\ran P_{H_\Gamma} (E) \subset \cD (\mathcal{F}_t)$ for $E \in \RR$, where
$\cD (\mathcal{F}_t)$ denotes the domain of $\mathcal{F}_t$.
For $\psi \in \ran P_{H_\Gamma} (E)$ and $T > 0$ we define the function $\Psi \colon \Gamma \times (-T,T) \to \CC$ as
\[
 \Psi(x,t)
 =
 (\mathcal{F}_t \psi)(x) .
\]
Note that $\Psi (\cdot , t) \in L^2 (\Gamma)$ for all $t \in (-T,T)$. Moreover, we define the (non-self-adjoint) operator $\hat H_\Gamma$ on
$L^2 (\Gamma \times (-T,T)) \cong L^2((-T,T),L^2(\Gamma))$
on
\[
 \cD(\hat H_\Gamma) = \left\{ \Phi\in L^2((-T,T),L^2(\Gamma)) \colon
 t \mapsto H_\Gamma(\Phi(t)) - (\frac{\partial^2}{\partial t^2} \Phi)(t) \in L^2((-T,T),L^2(\Gamma))\right\}
\]
by
\[
 \hat H_\Gamma = -\Delta + \hat V ,
 \quad \text{where} \quad
 \hat V (x,t) = V (x) .
\]
Here, $\Delta$ denotes the $d+1$-dimensional Laplacian.
We formulate a special case of Lemma 2.5 in \cite{NakicTTV-18-arxiv}.
\begin{lemma} \label{lemma:properties_Psi}
For all $T > 0$, $E \in \RR$ and all $\psi \in \ran P_{H_\Gamma} (E)$ we have:
\begin{enumerate}[(i)]
 \item The map $(-T,T) \ni t \mapsto \Psi (\cdot , t) \in L^2 (\Gamma)$  is infinitely $L^2$-differentiable with
 \begin{equation*}
  \Bigl(\frac{\partial}{\partial t}   \Psi \Bigr) (\cdot , 0)
  = \psi.
 \end{equation*}
 \item
 $\Psi \in \cD (\hat H_\Gamma)$ and $\hat H_\Gamma \Psi = 0$.
\end{enumerate}
\end{lemma}
From Lemma \ref{lemma:properties_Psi} part (ii) we infer that $\Psi$ is an eigenfunction of $\hat H_\Gamma$. This allows us to apply
similar techniques to the function $\Psi$ as used in the proofs of the results presented in Subsection~\ref{sec:eigenfunction}.
In order to recover properties of $\psi$ from properties of $\Psi$ one combines a second Carleman estimate with boundary terms
already used in \cite{LebeauR-95,JerisonL-99}
with part (i) of Lemma~\ref{lemma:properties_Psi}.

\section{From uncertainty to control}
\label{sec:uncertainty-control}

We introduce the notion of (null-)controllability in an abstract setting.
Let $\cH$ and $\cU$ be Hilbert spaces, $A$ a lower semi-bounded, self-adjoint operator in $\cH$ and $B$ a bounded operator from
$\cU$ to $\cH$.
Given $T>0$, we consider the abstract, inhomogeneous Cauchy problem
\begin{equation}
\label{eq:abstract_control_system}
 \begin{cases}
  \frac{\partial}{\partial t} u(t) + A u(t) &= B f(t),\ t \in (0,T],\\
  u(0) &= u_0 \in \cH,
 \end{cases}
\end{equation}
where $u \in L^2((0,T), \cH)$ and $f \in L^2((0,T), \cU)$.
The function $f$ is also called~\emph{control function} or simply~\emph{control} and the operator $B$ is
called~\emph{control operator}.
The~\emph{mild solution} to~\eqref{eq:abstract_control_system} is given by the Duhamel formula
\begin{equation}
 \label{eq:Duhamel_formula}
 u(t)
 =
 \euler^{-t A} u_0 + \int_0^t \euler^{-(t-s) A} B f(s) \drm s, \quad t \in [0,T].
\end{equation}
One central question in control theory is whether, given an input state $u_0$, a time $T > 0$ and a target state $u_T$, it is
possible to find a control $f$, such that $u(T) = u_T$.

\begin{definition}
 \label{def:null_controllable}
 Let $T > 0$. The system~\eqref{eq:abstract_control_system} is~\emph{null-controllable in time $T$} if for every $u_0 \in \cH$ there exists a
 control $f = f_{u_0} \in L^2((0,T) , \cU)$ such that the solution of~\eqref{eq:abstract_control_system} satisfies $u(T) = 0$.
 In this case the function $f_{u_0} $ will be called a \emph{null-control} in time $T$ for the initial state $u_0$.
\par
The~\emph{controllability map} or~\emph{input map} is the mapping $\cB^T \colon  L^2((0,T), \cU) \to \cH$ given by
\[
 \cB^T f = \int_0^T \euler^{-(T-s) A} B f(s) \drm s .
 \]
 \end{definition}
 
Taking into account~\eqref{eq:Duhamel_formula}, clearly a function $f$ is a null-control for~\eqref{eq:abstract_control_system} if and
only if $\euler^{-TA} u_0 + \cB^T f = 0$. Thus, the system ~\eqref{eq:abstract_control_system} is null-controllable in time $T > 0$
if and only if one has the relation $\ran \cB^T \supset \ran \euler^{-TA}$, which gives an alternative definition of
null-controllability in terms of the controllability map.

\begin{remark}
 Note that if the system~\eqref{eq:abstract_control_system} is null-controllable in time $T > 0$, then, by linearity of
 $\euler^{- T A}$, it is also controllable on the range of $\euler^{- T A}$. This means that for every $u_0 \in \cH$ and every
 $u_T \in \ran \euler^{- T A}$ there is a control $f \in L^2((0,T), \cU)$ such that the solution
 of~\eqref{eq:abstract_control_system} satisfies $u(T) = u_T$.
\end{remark}

In the context of the heat equation on a compact, connected and smooth manifold with control operator $B = \chi_{S}$,
null-controllability was proved for all $T > 0$ in~\cite[Theorem~1]{LebeauR-95} and independently in~\cite{FursikovI-96}:

\begin{theorem}
 Let $\cH = \cU = L^2(\Omega)$ for a compact and connected $C^\infty$ manifold $\Omega$, $A = - \Delta$ and $B = \chi_S$ for some
 non-empty, open $S \subset\Omega$, and $T > 0$.
 Then, the system~\eqref{eq:abstract_control_system} is null-controllable in time $T$.
\end{theorem}
In fact, the statement in~\cite{LebeauR-95} is stronger since it allows for the control set $S$ to change in time and it states that
the null-control can be chosen smooth and with compact support.
\par
The concept of null-controllability is closely related to a second one, the so-called~\emph{final-state-observability}:
For $T > 0$ we consider the homogeneous system
\begin{equation}
 \label{eq:abstract_control_system_free}
  \begin{cases}
  \frac{\partial}{\partial t} u(t) + A u(t) &= 0,\ t\in (0,T],\\
  u(0) &= u_0 \in \cH
 \end{cases}
\end{equation}
with solution given by $u(t) = \euler^{- A t} u_0$ for $t \in [0,T]$.

\begin{definition}
 The system~\eqref{eq:abstract_control_system_free} is called \emph{final-state-observable} in time $T > 0$ if there is a constant
 $C_\obs > 0$ such that for all $u_0 \in \cH$ we have
 \begin{equation}
  \label{eq:abstract_observability_estimate}
  \lVert \euler^{- A T} u_0 \rVert_\cH^2
  \leq
  C_\obs^2
  \int_0^T
  \lVert B^\ast \euler^{- A t} u_0 \rVert_\cU^2
  \drm t
 \end{equation}
 with $B$ from~\eqref{eq:abstract_control_system}. Ineq.~\eqref{eq:abstract_observability_estimate} is
 called~\emph{observability inequality}.
\end{definition}

In~\cite[Corollary 2]{LebeauR-95}, it is noted that null-controllability of the system~\eqref{eq:abstract_control_system} leads to
final-state-observability of~\eqref{eq:abstract_control_system_free}.
In fact, it is known that the notions of null-controllability and final-state-observability are equivalent:
\begin{theorem}[{\cite{Russell78}, see also \cite[Chapter IV.2.]{Zabczyk-08}}]
\label{thm:observabilty_controllability}
 Let $T > 0$.
 The system~\eqref{eq:abstract_control_system} is null-controllable in time $T$ if and only if the
 system~\eqref{eq:abstract_control_system_free} is final-state-observable in time $T$.
\end{theorem}
Theorem~\ref{thm:observabilty_controllability} is, in fact, a direct consequence of the following lemma, 
which is a well-known result going back to~\cite{Douglas-66}, see also~\cite{DoleckiR-77,CurtainP78, Lions-88}. 
The proof given here is inspired by the corresponding proofs in~\cite{Zabczyk-08} and~\cite[Proposition 12.1.2]{TucsnakW09}.
\begin{lemma}
\label{lem:absop}
  Let $\cH_1, \cH_2,\cH_3$ be Hilbert spaces, and let $X\colon \cH_1 \to \cH_3$, $Y\colon \cH_2 \to \cH_3$ be bounded operators.
  Then, the following are equivalent:
  \begin{enumerate}[(a)]
   \item $\ran X \subset \ran Y$;
   \item There is $c>0$ such that $\lVert X^* z \rVert \le c \lVert Y^* z \rVert $ for all $z\in \cH_3$.
   \item There is a bounded operator $Z\colon \cH_1 \to \cH_2$ satisfying $X=YZ$.
  \end{enumerate}
  Moreover, in this case, one has
  \begin{equation}\label{eq:absop}
   \inf\{ c\colon c\text{ as in (b)}\} = \inf\{\norm{Z} \colon Z \text{ as in (c)}\},
  \end{equation}
  and both infima are actually minima.
\end{lemma}
\begin{proof}
 (a)$\Rightarrow$(b). First, suppose that $\Ker Y = \{ 0 \}$. Let $\tilde \cH_3 = \overline{\ran Y}$ be the Hilbert space with the
 same scalar product as in $\cH_3$. Then, we can regard $X$, $Y$ as operators with the codomains $\tilde \cH_3$ and
 $Y^{-1}\colon \tilde\cH_3\to \cH_2$ exists and is densely defined. The operator $Y^{-1}X$ is an everywhere defined closed operator,
 hence bounded by the closed graph theorem. In turn, also $(Y^{-1}X)^{*}$ is bounded. From \cite[Proposition 1.7]{Schmuedgen-12} it
 follows that $X^{(*)}Y^{-(*)} \subset (Y^{-1}X)^{*}$, where $(*)$ denotes the adjoint with respect to $\tilde \cH_3$. Hence there
 exists $c\ge 0$ such that $\lVert X^{(*)}z \rVert  \leq c \lVert Y^{(*)}z \rVert $ for all $z\in \tilde \cH_3$. But it is easy to
 see that $X^{(*)} z = X^* z$ and $Y^{(*)} z = Y^* z$ for all $z \in \overline{\ran Y}$. Finally note that if
 $z\in \left( \overline{\ran Y} \right)^{\perp} $ then $Y^* z = 0 = X^*z$ since
 $\Ker Y^* = \left( \ran Y \right)^\perp \subset \left( \ran X \right)^\perp = \Ker X^*$ by hypothesis. Hence, in this case, (b) is
 proved.

 If $\Ker Y$ is not trivial, instead of $Y$ we take $\hat Y$, the restriction of $Y$ to the space $(\Ker Y)^{\perp}$. Since
 $\ran \hat Y = \ran Y$, we can apply the first part of the proof to show $\lVert X^{*}z \rVert  \leq c \lVert \hat Y^{*}z \rVert $
 for all $z\in \cH_3$. Since $\hat Y^* z = Y^* z$ for all $z\in \cH_3$, the claim follows.

 (b)$\Rightarrow$(c). We define the operator $K\colon \ran Y^* \to \ran X^*$ by $K (Y^* z) = X^* z$ for all $z\in \cH_3$. The
 hypothesis implies that $K$ is well defined and bounded with norm less or equal to $c$. We continuously extend $K$ to
 $\overline {\ran Y^*}$ and by zero to a bounded operator on $\cH_2$. Then still $\|K\|_{\cH_2\to \cH_1} \leq c$. We obviously have
 $K Y^* = X^*$ by construction and hence also $X = Y K^*$, which implies the claim with $Z=K^*$. Since $\|Z\| = \|K\| \leq c$, this also shows that
 the right-hand side of~\eqref{eq:absop} does not exceed the left-hand side.

 (c)$\Rightarrow$(b). We clearly have
 \begin{equation*}
  \lVert X^* z \rVert = \Vert Z^* Y^* z \rVert \le \lVert Z^* \rVert \cdot \lVert Y^* z \rVert
 \end{equation*}
 for all $z\in \cH_3$, which proves the claim with $c = \lVert Z^* \rVert = \lVert Z \rVert$. This also shows that the left-hand
 side of~\eqref{eq:absop} does not exceed the right-hand side.

 (c)$\Rightarrow$(a). This is obvious.

 This concludes the proof of the equivalence of (a)--(c) and also of the identity~\eqref{eq:absop}. It remains to show that both
 minima in~\eqref{eq:absop} are actually minima. This is clear for the infimum on the left-hand side. In turn, it then follows
 from~\eqref{eq:absop} and the proof of (b)$\Rightarrow$(c) that also the infimum on the right-hand side is a minimum, which
 completes the proof.
\end{proof}%

\begin{proof}[Proof of Theorem~\ref{thm:observabilty_controllability}]
 Observe that
 \begin{equation*}
  \int_0^T \lVert B^\ast \euler^{- A t} u_0 \rVert_\cU^2 \drm t = \lVert (\cB^T)^* u_0 \rVert_{L^2((0,T), \cU)}^2.
 \end{equation*}
 The claim therefore follows from the equivalence between (a) and (b) in Lemma~\ref{lem:absop} by taking
 $X = \euler^{-AT}\colon \cH \to \cH$ and $Y = \cB^T \colon L^2((0,T), \cU) \to \cH$.
\end{proof}%
Lemma~\ref{lem:absop} actually gives much more information:
If the system \eqref{eq:abstract_control_system}
 is null-controllable, corresponding to case (a) in the Lemma, there exists according to case (c) a bounded operator $\cF\colon \cH \to L^2((0,T), \cU)$ such that
\begin{equation}
\label{eq:F-norm}
\|\cF\| = \min \left\{ c>0 \mid \forall z \in \cH : \|\euler^{-TA}z\|_{\cH} \leq c \|(\cB^T)^*z\|_{ L^2((0,T), \cU)}=\int_0^T \lVert B^\ast \euler^{- A t} z \rVert_\cU^2 \drm t \right\}
\end{equation}
and $\euler^{-TA} + \cB^T\cF=0$. 
In particular, $\cF u_0$  is a null-control in time $T$ for an initial state $u_0\in \cH$. 
Moreover, if we fix an initial datum $u_0$ and a time $T>0$, and are given one particular null-control $f_0$, the set of all null-controls 
is given by the closed affine space
  \[
   f_0 + \Ker \cB^T.
  \]
If $P$ denotes the orthogonal projection onto $\Ker \cB^T$ we have $-\euler^{-AT} = \cB^T (I-P)\cF$ and the operator $\cF^T:=(I- P) \cF$ does not depend on the choice of $\cF$.
  It follows that for every $u_0 \in \cH$, the function $\cF^T u_0 \in L^2((0,T), \cU)$ is the unique control with minimal norm associated to the initial datum $u_0$.
 
  Together with the identity~\eqref{eq:F-norm}, this justifies the following definition.

\begin{definition}\label{def:optimal-feedback}
 If the system~\eqref{eq:abstract_control_system} is null-controllable, then 
the norm of the above defined optimal operator $\cF^T \colon \cH \to L^2((0,T), \cU)$ is called \emph{control cost in time $T$}.
It satisfies
 \begin{align*} 
  C_T
  :=
  \lVert \cF^T \rVert
  &=
  \sup_{\lVert u_0 \rVert_{\cH} = 1} \min \{ \lVert f \rVert_{L^2((0,T),\cU)} \colon \euler^{-TA} u_0 + \cB^T f = 0 \}
  \\   
  &=
  \min
  \{
  C_\obs \colon \text{$C_\obs$ satisfies }\eqref{eq:abstract_observability_estimate} \}.
 \end{align*}
\end{definition}
The equivalence between final-state-observability and null-controllability can be seen as a way to reduce the study of properties of the inhomogeneous system (null-controllabili\-ty) to properties of the homogeneous system (final-state-observability).

A crucial ingredient for proving observability estimates are \emph{uncertainty relations}.
An uncertainty relation is an estimate of the form
\begin{align}
 \label{eq:abstract_spectral_inequality}
 \forall E \in \RR, u \in \cH
 \colon
 \quad
 \lVert \chi_{(- \infty, E]} (A) u \rVert_\cH^2
 \leq
 C_{\ur}(E)
 \lVert B^\ast \chi_{(- \infty, E]} (A) u \rVert_\cU^2
\end{align}
for some function $C_\ur \colon \RR \to [0, \infty)$.
As we will see below, in the context of interest to us, it is possible to prove estimates of this type with
\begin{equation}
 \label{eq:constant_in_uncertainty_relation}
 C_\ur (E)
 =
 d_0 \euler^{d_1 E_+^s}
\end{equation}
for some $s \in (0,1)$ and constants $d_0, d_1 > 0$. Recall that $t_+ = \max\{0,t\}$ for $t \in \RR$.

In the case of the pure Laplacian, such estimates can be deduced from the Logvinenko-Sereda theorem, cf.~Corollaries~\ref{cor:spectral-inequality-full-space} and~\ref{cor:cubes-spectral}.
In the case of Schr\"odinger operators, they can be proved by means of Carleman estimates as discussed in Section~\ref{sec:Carleman}.

\begin{remark}[Terminology]
In the case where $A$ is an elliptic second order differential operator (on a subset of $\RR^d$ or on a manifold) and $B$ is the indicator function of a non-empty, open subset,  Ineq.~\eqref{eq:abstract_spectral_inequality} is also referred to as a \emph{quantitative unique continuation principle}.
In the context of control theory, it is also called \emph{spectral inequality}.
\end{remark}
In~\cite{RousseauL-12}, a very transparent interplay between null-controllability, final-state-obser\-vability and spectral inequalities is used to iteratively construct a null-control
and thus establish null-controllabi\-li\-ty.
Since this approach is very instructive in nature, we are going to present their strategy in detail here.
Even though in~\cite{RousseauL-12} the special case
of the heat equation on bounded domains $\Omega$ with $B = \chi_S$ for some open $S\subset\Omega$
has been considered, we formulate their proof here in an abstract setting.
In particular, it does not require the operator $A$ to have purely discrete spectrum and thus can also be applied for the heat equation on unbounded domains, provided that a corresponding spectral inequality has been established.
\begin{theorem}
 \label{thm:control_active-passive}
 Assume that $A \geq 0$ is a self-adjoint operator and that the spectral inequality~\eqref{eq:abstract_spectral_inequality} holds for $E \geq 0$ with $C_\ur(E) = C \euler^{C \sqrt{E}}$ for some $C \geq 1$.
 Then, for every $T > 0$ the system~\eqref{eq:abstract_control_system} is null-controllable.
\end{theorem}

 The main idea in the proof of Theorem~\ref{thm:control_active-passive} in~\cite{RousseauL-12} are so-called \emph{active and passive phases}.
 For that purpose, the time interval is decomposed $[0,T] = \cup_{j \in \NN_0} [a_j, a_{j+1}]$ where $a_0 = 0$, $a_{j+1} = a_j + 2 T_j$ for $T_j > 0$ to be specified in the proof, and with $\lim_{j \to \infty} a_j = T$.
 The subintervals $[a_j, a_j + T_j]$ are called \emph{active phases} and the subintervals $[a_j + T_j, a_{j+1}]$ \emph{passive phases}.
 The idea is now to choose a sequence $(E_j)_{j \in \NN_0}$, tending to infinity and to split for every $j \in \NN_0$ the system according to $\cH = \ran \chi_{(- \infty, E_j]}(A) \oplus \ran \chi_{(E_j, \infty)}(A)$ into a \emph{low energy} and a \emph{high energy part}.
 In every active phase $[a_j, a_j + T_j]$, one then deduces final-state-observability of the low energy part $\ran \chi_{(- \infty, E_j]}(A)$ and thus finds a control in this time interval such that at time $a_j + T_j$, the solution will be in $\ran \chi_{(E_j, \infty)}(A)$, i.e.\ it will be in the high energy part of the state space.
 Then, in the passive phase, no control will be applied and by contractivity of the semigroup $\euler^{- A t}$, the solution will decay proportional to $\euler^{- T_j E_j}$.
 Repeating this procedure, we will see that with appropriate choices of the $T_j$ and the $E_j$, the solution tends to zero as $j \to \infty$, i.e.\ as $t \to T$.

 In order to make these ideas more precise, the following energy-truncated control system is introduced:
\begin{equation}
 \label{eq:abstract_control_system_truncated}
 \begin{cases}
   \frac{\partial}{\partial t} v(t) + A v(t) &= \chi_{(- \infty, E]}(A) B f(t),\\
   v(0) &= v_0 \in \ran \chi_{(- \infty, E]}(A).
 \end{cases}
\end{equation}

\begin{lemma}
 \label{lem:control_truncated_system}
 Let $\cT > 0$ and assume that the spectral inequality~\eqref{eq:abstract_spectral_inequality} holds for all $E \geq 0$.
 Then for every $E \geq 0$, the system~\eqref{eq:abstract_control_system_truncated} is null-controllable in time $\cT$ with cost $C_\cT$ satisfying $C_\cT^2 = C_\ur(E)/\cT$.
\end{lemma}

\begin{proof}
 It suffices to see that the system
 \[
  \begin{cases}
   \frac{\partial}{\partial t} v(t) + A v(t) &= 0,\ t>0\\
   v(0) &= v_0 \in \ran \chi_{(- \infty, E]}(A),\\
  \end{cases}
 \]
 considered as a system on the Hilbert space $\ran \chi_{(- \infty, E]}(A)$, is final-state-observable in time $\cT$.
 For that purpose, we calculate, using spectral calculus and in particular the fact that $\euler^{A t}$ leaves $\ran \chi_{(- \infty, E]}(A)$ invariant, and~\eqref{eq:abstract_spectral_inequality} that
 \begin{align*}
  \cT \lVert \euler^{-A \cT} v_0 \rVert_\cH^2
  &\leq
  \int_0^\cT \lVert \euler^{-A t} v_0 \rVert_\cH^2 \drm t
  \leq
 C_\ur(E)
  \int_0^\cT \lVert B^\ast \euler^{-A t} v_0 \rVert_\cU^2 \drm t.
  \qedhere
 \end{align*}
\end{proof}

\begin{proof}[Proof of Theorem~\ref{thm:control_active-passive}]
 Following~\cite[Section 6.2]{RousseauL-12}, we split the time interval $[0,T] = \cup_{j \in \NN_0} [a_j, a_{j+1}]$ with $a_0 = 0$, $a_{j+1} = a_j + 2 T_j$, and $T_j = K 2^{-j /2}$ for a constant $K$ defined by the relation $2 \sum_{j = 0}^\infty T_j = T$.
 Furthermore, we choose $E_j = 2^{2 j}$.

 Our aim is to choose in every active phase $[a_j, a_j + T_j]$ an appropriate null-control $f_j \in L^2([a_j, a_j + T_j], \cU)$ such that $u(a_j + T_j) \in \ran \chi_{(E_j, \infty)}(A)$.

 Therefore, let $j \in \NN_0$ and $u(a_j) \in \cH$ be given.
 In the active phase $[a_j, a_j + T_j]$, we apply Lemma~\ref{lem:control_truncated_system} with $v_0 = \chi_{(- \infty, E_j]}(A) u(a_j)$, and $\cT = T_j$.
 This yields a function $f_j \in L^2([a_j, a_j + T_j], \cU)$ with
 \[
  \int_{a_j}^{a_j + T_j} \lVert f_j(t) \rVert_\cU^2 \drm t
  \leq
  \frac{C \euler^{C 2^j}}{T_j}
  \lVert u(a_j) \rVert_\cH^2,
 \]
 such that the solution of the system
 \[
 \begin{cases}
  \frac{\partial}{\partial t} v(t) + A v(t) &= \chi_{(- \infty, E_j]}(A) B f_j(t),
  \quad
  t\in(a_j, a_j + T_j),\\
  v(a_j)
  &=
  \chi_{(- \infty, E_j]}(A) u(a_j)
  \in \ran \chi_{(- \infty, E_j]}(A),\\
 \end{cases}
 \]
 satisfies $v(a_j + T_j) = 0$. Since the spectral projectors of $A$ commute with $\euler^{- t A}$, with this control function $f_j$
 in $(a_j,a_j+T_j]$ we then have $\chi_{(-\infty,E_j]}(A)u(a_j+T_j)=0$ and
 \begin{align}
  u(a_j + T_j)
  &=
  \chi_{(E_j, \infty)}(A) u(a_j + T_j)
  \nonumber
  \\
  &=
  \euler^{- T_j A} \chi_{(E_j, \infty)}(A) u(a_j)
  +
  \int_{a_j}^{a_j + T_j} \euler^{- (a_j + T_j - t) A} \chi_{(E_j, \infty)}(A) B f_j(t) \drm t.
  \label{eq:Duhamel_after_splitting}
 \end{align}
 We use the notation $F(t) := \euler^{- (a_j + T_j - t) A} \chi_{(E_j, \infty)}(A) B f_j(t)$ and estimate
 \begin{align*}
  \Bigl\lVert
    \int_{a_j}^{a_j + T_j}
    &\euler^{- (a_j + T_j - t) A}
    \chi_{(E_j, \infty)}(A) B f_j(t) \drm t
  \Bigr\rVert_\cH^2
  \leq
    \int_{a_j}^{a_j + T_j}
    \int_{a_j}^{a_j + T_j}
  \lVert
    F(t)
  \rVert_\cH
  \cdot
  \lVert
   F(s)
  \rVert_\cH
  \drm t
  \drm s
  \\
  &\leq
  \frac{1}{2}
    \int_{a_j}^{a_j + T_j}
    \int_{a_j}^{a_j + T_j}
  \lVert
    F(t)
  \rVert_\cH^2
  \drm t
  \drm s
  +
  \frac{1}{2}
  \int_{a_j}^{a_j + T_j}
  \int_{a_j}^{a_j + T_j}
  \lVert
   F(s)
  \rVert_\cH^2
  \drm t
  \drm s
  \\
  &=
  \int_{a_j}^{a_j + T_j}
  \int_{a_j}^{a_j + T_j}
  \lVert
    F(t)
  \rVert_\cH^2
  \drm t
  \drm s
  \leq
  T_j \lVert B \rVert^2 \frac{C \euler^{C 2^j}}{T_j} \lVert u(a_j) \rVert_\cH^2.
 \end{align*}
 Hence, we obtain from~\eqref{eq:Duhamel_after_splitting} and using that $C \geq 1$
 \begin{align*}
  \lVert u(a_j + T_j) \rVert_\cH
  &\leq
  \left(
  1
  +
  \lVert B \rVert \sqrt{C} \euler^{(C/2) 2^j}
  \right)
  \lVert u(a_j) \rVert_\cH
  \\
  &\leq
  \euler^{(2 + \lVert B \rVert) C 2^j}
  \lVert u(a_j) \rVert_\cH
  =:
  \euler^{D 2^j} \lVert u(a_j) \rVert_\cH.
 \end{align*}
 Now, using $u(a_j + T_j) \in \ran \chi_{(E_j, \infty)}(A) = \ran \chi_{(2^{2j}, \infty)}(A)$ and recalling that $T_j = K 2^{-j/2}$, we find
 \begin{align*}
 \lVert u(a_{j+1}) \rVert_\cH
 &\leq
 \euler^{- 2^{2j} T_j} \lVert u(a_j + T_j) \rVert_\cH\\
 &\leq
 \euler^{D 2^j - K 2^{3 j /2 } } \lVert u (a_j) \rVert_\cH.
 \end{align*}
 Inductively, this yields
 \[
  \lVert u(a_{j+1}) \rVert_\cH
  \leq
  \exp \left( \sum_{k = 0}^j D 2^k - K 2^{3 k/2} \right) \lVert u(0) \rVert_\cH.
 \]
 Thus, $\lim_{j \to \infty} \lVert u(a_j) \rVert_\cH^2 = 0$.
 It remains to show that the function $f \colon [0, T] \to \cU$, defined by
 \[
  f(t)
  :=
  \begin{cases}
   f_j(t) & \text{if $t \in [a_j, a_j + T_j]$,}\\
   0    & \text{else}
  \end{cases}
 \]
 is in $L^2((0,T), \cU)$.
 For that purpose, we calculate
 \begin{align}
  \lVert f \rVert_{L^2((0,T), \cU)}^2
  &=
  \sum_{j = 0}^\infty
  \int_{a_j}^{a_j + T_j} \lVert f_j(t) \rVert_\cH^2 \drm t
  \leq
  \sum_{j = 0}^\infty
  \frac{C \euler^{C  2^j}}{T_j} \lVert u(a_j) \rVert_\cH^2
  \nonumber
  \\
  &\leq
  \left(
  \frac{C \euler^{2 C}}{T_0}
  +
  \sum_{j = 1}^\infty
  \frac{C \euler^{C 2^j}}{T_j}
  \exp \left( \sum_{k = 0}^{j - 1} 2 D 2^k - 2 K 2^{3 k/2} \right)
  \right)
  \lVert u(0) \rVert_\cH^2
  \nonumber
  \\
  &=
  \left(
  \frac{C \euler^{2 C}}{T_0}
  +
  \sum_{j = 1}^\infty
  \frac{C}{K}
  \exp \left( C 2^j + \frac{\ln(2) j}{2} + \sum_{k = 0}^{j - 1} 2 D 2^k - 2 K 2^{3 k/2} \right)
  \right)
  \lVert u(0) \rVert_\cH^2,
  \label{eq:series_norm_control_function_LebeauLeRousseau}
 \end{align}
 and since there are $\tilde C_1, \tilde C_2 > 0$ such that
 \begin{align*}
  C 2^j
  + \frac{\ln(2) j}{2}
  + \sum_{k = 0}^{j - 1} 2 D 2^k - 2 K 2^{3 k/2}
  &=
  C 2^j + \frac{\ln(2) j}{2}
  +
  2 D \frac{2^j - 1}{2-1}
  -
  2 K
  \frac{2^{3j/2} - 1}{2^{3/2} - 1}
  \\
  &\leq
  \left( C + \frac{\ln(2)}{2} + 2 D + \frac{2K}{2^{3/2}-1}\right)
  2^j
  -
  \left( \frac{2 K}{2^{3/2}} \right)
  2^{3j/2}
  \\
  &\leq
  \tilde C_1 - \tilde C_2 2^j
  \quad
  \text{for all $j \in \NN$},
 \end{align*}
 the series in~\eqref{eq:series_norm_control_function_LebeauLeRousseau} converges.
 This concludes the proof.
\end{proof}

We have now seen how a spectral inequality leads to null-controllability.
While being very constructive in nature, the above method makes it challenging to keep track of the estimate on the control cost, that is, on the norm of the null-control $f$, in terms of model parameters.
Even trying to understand its $T$-dependence is difficult.
This becomes even more involved if we endow the spectral inequality with more constants, e.g.\ by choosing $C_\ur(E) = d_0 \euler^{d_1 \sqrt{E}}$, and attempt to also understand the dependence of the control cost in terms of $d_0$ and $d_1$.

However, there exist other works which have derived more explicit bounds on the control cost.
There, usually an observability estimate for the whole system is proved without going through the active-passive-phases construction.
The first work we cite here is~\cite{Miller-10}, where ideas of~\cite{LebeauR-95} have been streamlined and generalized to a more abstract situation.
In fact, Miller considered a situation where the operator $A$ is no longer self-adjoint, but merely the generator of a strongly continuous semigroup.
Due to the lack of spectral calculus, an additional assumption on contractivity of the semigroup on certain invariant subspaces \eqref{eq:contractivity} is required and serves as a replacement for the strict contractivity of the semigroup on high energy spectral subspaces.
Furthermore, the situation is treated where the spectral inequality holds for an additional reference operator $B_0$ which is in some relation to the actual control operator $B$ (actually, it will be the identity operator in our applications below).
\begin{theorem}[{\cite[Theorem~2.2]{Miller-10}}]
 \label{thm:Miller}
 Let a (not necessarily self-adjoint) operator $-A$ in $\cH$ be the generator of a
 strongly continuous semigroup $\{ \euler^{-t A} \colon t \geq 0 \}$.
 Assume that there is a family $\cH_\lambda \subset \cH$, $\lambda > 0$, of semigroup invariant subspaces such that for some $\nu \in (0,1)$, $m \geq 0$, $m_0 \geq 0$, and $T_0 > 0$ we have
 \begin{equation}
 \label{eq:contractivity}
  \forall
  \lambda > 0,
  x \in \cH_\lambda^\perp,
  t \in (0, T_0),
  \quad
  \lVert \euler^{-tA} x \rVert_\cH
  \leq
  m_0
  \euler^{m \lambda^\nu}
  \euler^{- \lambda t}
  \lVert x \rVert_\cH
  .
 \end{equation}
 Let $B_0$ be an operator, mapping from $\cD(A)$ to $\cU$, satisfying
 \begin{equation}
 \label{eq:spec_ineq_miller}
  \forall x \in \cH_\lambda,
  \lambda > 0,
  \quad
  \lVert B_0 x \rVert_{\cH}^2
  \leq
  a_0 \euler^{2 a \lambda^\alpha}
  \lVert B x \rVert_{\cH}^2
 \end{equation}
 for some $a_0, a, \alpha > 0$.
 Assume that there are $b_0$, $\beta$, $b > 0$ such that
 \begin{equation}
 \label{eq:obs_ineq_B0}
  \forall x \in \cD(A),
  T \in (0, T_0),
 \quad
  \lVert \euler^{-T A} x \rVert_{\cH}^2
  \leq
  b_0 \euler^{\frac{2 b }{T^\beta}}
  \int_0^T
  \lVert B_0 \euler^{-tA} x \rVert_{\cH}^2 \drm t.
 \end{equation}
 Assume that we can choose $\beta = \frac{\alpha}{1 - \alpha} = \frac{\nu}{1 - \nu}$.

 Then, for all $T > 0$, we have the observability estimate
 \[
  \lVert \euler^{-T A } x \rVert_{\cH}^2
  \leq
  \kappa_T
  \int_0^T \lVert B \euler^{-tA} x \rVert_{\cH}^2
  \drm t,
  \quad
  \forall x \in \cD(A)
 \]
 where $\kappa_T$ satisfies $2 c = \limsup_{T \to 0} T^\beta \ln \kappa_T < \infty$ with the constant $c$ satisfying
 \[
  c
  \leq
  c_\ast
  =
  \left( \frac{(\beta + 1) b}{a + m} \right)^{\frac{\beta + 1}{\beta}}
  \frac{\beta^\beta}{s^{\frac{(\beta + 1)^2}{\beta}}}
 \]
 with
 \[
  s(s + \beta + 1)^\beta
  =
  (\beta + 1) \beta^{\frac{\beta^2}{\beta + 1}}
  \frac{b^{\frac{1}{\beta + 1}}}{a + m}.
 \]
 Moreover, if we have
 \[
  \forall
  x \in \cD(A),
  T > 0 ,
  \quad
  \int_0^T
  \lVert B \euler^{-t A} \rVert_{\cH}^2 \drm t
  \leq
  \adm_T \lVert x \rVert_{\cH}^2
 \]
 with a constant $\adm_T$ satisfying $\lim_{T \to 0} \adm_T = 0$, then there exists $T' > 0$ such that for all $T \in (0, T']$, we have
 \[
  \kappa_T
  \leq
  4 a_0 b_0
  \exp \left( \frac{2 c_\ast}{T^\beta} \right).
 \]
\end{theorem}
In particular, the control cost $\kappa_T$ is estimated only for sufficiently small times.

One can apply Theorem \ref{thm:Miller} in various ways.
For instance, it is possible to choose $B_0 = I$, in which case \eqref{eq:obs_ineq_B0} is obviously satisfied for small times and \eqref{eq:spec_ineq_miller} becomes a spectral inequality. Depending on the system, the latter can be challenging to establish or not.
Alternatively, one might be able to prove \eqref{eq:obs_ineq_B0} for a convenient operator $B_0$ for which \eqref{eq:spec_ineq_miller} is easier to establish.

In the case of the system \eqref{eq:abstract_control_system}, Theorem \ref{thm:Miller} simplifies to the following result.
\begin{corollary}
Let $A\geq 0$ be a self-adjoint operator in a Hilbert space $\cH$ and $B \in \cL(\cU, \cH)$.
Then, \eqref{eq:contractivity} holds with  $m_0=1$ and $ m=0$.
Let \eqref{eq:abstract_spectral_inequality} be satisfied with $C_{\ur}(\lambda) = a_0 \euler^{2 a \lambda^\alpha}$
(i.e.\ \eqref{eq:spec_ineq_miller} is valid for $B_0=I$ and $\cH_\lambda$ being spectral subspaces of $A$ corresponding to the interval $(- \infty, \lambda]$).
Then \eqref{eq:obs_ineq_B0} is satisfied for any choice of $b, b_0 > 0$, provided $T_0$ is small enough.
Consequently the conclusions of Theorem \ref{thm:Miller} hold true.
\end{corollary}

In the particular case where the spectral inequality~\eqref{eq:abstract_spectral_inequality} with $C_\ur(E)$ as in~\eqref{eq:constant_in_uncertainty_relation} and $s = 1/2$ holds, the result of~\cite{Miller-10} implies that the system~\eqref{eq:abstract_control_system_free} is final-state-observable in sufficiently small time $T$.
Thus the system~\eqref{eq:abstract_control_system} is null-controllable in time $T$ with cost satisfying
\begin{equation*}
 C_T
 \leq
 d_0 \exp \left( \frac{c_\ast}{T} \right)
 ,
 \quad
 0 < T \leq T'
\end{equation*}
for some $T', c_\ast > 0$, depending in an implicit manner on $d_0$ and $d_1$.
We emphasize that this result provides estimates on the control cost only for small times $0 < T \leq T'$, where $T'$ also depends in an implicit way on the model parameters.

In~\cite[Theorem 2.1]{BeauchardPS-18}, Beauchard, Pravda-Starov and Miller removed this restriction to small times in the specific situation where $\cH = L^2(\Omega)$ and $B = \chi_S$ for $S \subset \Omega \subset \RR^d$.
\begin{theorem}[{\cite[Theorem~2.1]{BeauchardPS-18}}]
\label{thm:BPSM}
 Let $\Omega$ be an open subset of $\RR^d$, $S$ be an open subset of $\Omega$, $\{ \pi_k \colon k \in \NN \}$ be a family of orthogonal projections on $L^2(\Omega)$, $\{ \euler^{-t A} \colon t \geq 0 \}$ be a contraction semigroup on $L^2(\Omega)$, $c_1$, $c_2$, $a$, $b$, $t_0$, $m >0$ be positive constants with $a < b$.
 If the spectral inequality
 \[
  \forall g \in L^2(\Omega),
  \forall k \geq 1,
  \quad
  \lVert \pi_k g \rVert_{L^2(\Omega)}
  \leq
  \euler^{c_1 k^a}
  \lVert \pi_k g \rVert_{L^2(S)},
 \]
 and the dissipation estimate
 \[
  \forall g \in L^2(\Omega),
  \forall k \geq 1,
  \forall 0 < t < t_0,
  \quad
  \lVert
  (1 - \pi_k)
  (\euler^{-T A} g)
  \rVert_{L^2(\Omega)}
  \leq
  \frac{1}{c_2}
  \euler^{- c_2 t^m k^b}
  \lVert g \rVert_{L^2(\Omega)}
 \]
 hold, then there exists a positive constant $C > 1$ such that the following observability estimate holds
 \[
  \forall T > 0,
  \forall g \in L^2(\Omega),
  \quad
  \lVert \euler^{-T A} \rVert_{L^2(\Omega)}^2
  \leq
  C
  \exp
  \left(
   \frac{C}{T^{\frac{am}{b-a}}}
  \right)
  \int_0^T
  \lVert \euler^{-tA} g \rVert_{L^2(S)}^2
  \drm t.
 \]
\end{theorem}

Let us remark that the proof of \cite[Theorem~2.1]{BeauchardPS-18} does not require $S$ to be open, but merely to have positive measure as observed in \cite{EgidiV-18}.

In the applications we discuss below, the projectors $\pi_k$ will be spectral projectors corresponding to the operator $A$ and the dissipation estimate will hold automatically.
Thus, the verification of the conditions of the theorem is again reduced to the verification of a spectral inequality.

In Theorem~\ref{thm:BPSM}, the estimate on the control cost is again given in the form
\begin{equation}
 \label{eq:control_cost_estimate_BeauchardPS}
 C_T = \tilde C \exp \left( \frac{\tilde C}{T} \right),
 \quad
 T > 0
\end{equation}
for a non-explicit constant $\tilde C$.
Note that this constant $C_T$ does not converge to zero as $T$ tends to $\infty$.
In some situations, however, the constant can be strengthened to show this asymptotic behavior at large times.
A step in this direction is~\cite[Theorem~1.2]{TenenbaumT-11}.
We note that there, more general control operators $B$ are considered, while $A$ is assumed to be a non-negative self-adjoint operator with purely discrete
spectrum.
\begin{theorem}[{\cite[Theorem~1.2]{TenenbaumT-11}}]
\label{thm:Tenenbaum-Tucsnac}
 Let $A$ be a non-negative operator in $\cH$ and let $B \in \cL(\cU, \cH_{\beta})$ for some $\beta \leq 0$, where $\cH_\beta$ is the completion of $\cH$ with respect to the scalar product
 \[
  \left\langle x,y \right\rangle_{\cH_\beta}
  =
  \left\langle (\Id + A^2)^{\beta/2} x, (\Id + A^2)^{\beta/2} x \right\rangle_\cH.
 \]
 Assume that $A$ is diagonalizable, that $\{ \phi_k \colon k \in \NN \}$ is an orthonormal basis of eigenvectors with corresponding non-decreasing sequence of eigenvalues $\{ \lambda_k \colon k \in \NN \}$ such that $\lim_{k \to \infty} \lambda_k = \infty$.
 Assume furthermore that there exists $s \in (0,1)$ such that for some $d_0, d_1 > 0$, we have
 \[
  \forall \{ a_k \}_{k \in \NN} \in \ell^2(\CC),
  \mu > 1,
  \quad
  \left(
  \sum_{\lambda_k^s \leq \mu}
  \lvert \alpha_k \rvert^2
  \right)^{1/2}
  \leq
  d_0 \euler^{d_1 \mu}
  \lVert
  \sum_{\lambda_k^s \leq \mu}
  a_k B^\ast \phi_k
  \rVert_\cU.
 \]
 Then, the system
 \begin{equation}
  \label{eq:system_TenenbaumTucsnac}
  \dot w = -A w + B u,
  \quad
  w(0) = z
 \end{equation}
 is null-controllable in any time $T > 0$.
 Moreover, given $c > h^{gh} g^{-g^2} d_1^h$, where $g = s/(1 - s)$, $h = g + 1 = 1/(1 - s)$, the control cost satisfies
 \[
  C_T
  \leq
  \tilde{C}
  T^{-1/2}
  \euler^{c/T^g}
 \]
 for a constant $\tilde{C}$ depending only on $d_0$, $d_1$, $c$, $\beta$, $s$, and $\lVert B \rVert_{\cL(\cU, \cH_{\beta})}$.
\end{theorem}

Note that in Theorem~\ref{thm:Tenenbaum-Tucsnac} as well as in Theorem~\ref{thm:NTTV_control_abstract} below, the Duhamel formula~\eqref{eq:Duhamel_formula}, defining the mild solution of the system~\eqref{eq:system_TenenbaumTucsnac}, now describes a function in the Hilbert space $\cH_\beta$ whence also the semigroup $\euler^{-A \cdot}$ needs to be appropriately extended from $\cH$ to $\cH_\beta$, see e.g.~\cite[II.5.a]{EngelN-99} for details.

Theorem~\ref{thm:Tenenbaum-Tucsnac} shows that if $A$ has compact resolvent, then a spectral inequality with $C_\ur = d_0 \euler^{d_1 \lambda^s}$ for all $\lambda \geq 0$ and some $s \in (0,1)$
implies null-controllability in all times $T > 0$ with cost satisfying
\begin{equation}
 \label{eq:control_cost_estimate_TenenbaumT}
 C_T
 \leq
 \frac{C_1}{\sqrt{T}} \exp \left( \frac{C_2}{T^{\frac{s}{1 - s}}} \right).
\end{equation}
The upper bound in~\eqref{eq:control_cost_estimate_TenenbaumT} decays proportional to $\sqrt{T}^{-1}$ as $T$ tends to infinity and thus improves upon the upper bound in~\eqref{eq:control_cost_estimate_BeauchardPS}.
Furthermore, \cite{TenenbaumT-11} provides an estimate on $C_2$ in terms of $s$ and $d_1$.
However, it remains unclear whether and how $C_1$ depends on $s$, $C_2$, $d_0$, $d_1$, and on the operator $B$.

While the results in~\cite{Miller-10} and in~\cite{BeauchardPS-18} are both inspired by~\cite{LebeauR-95} and thus the structure of the proofs is rather similar, the proof in~\cite{TenenbaumT-11} has a different structure which makes it easier to keep track of the dependence of the constant $C_T$ in terms of the model parameters, even though this analysis has not been thoroughly performed in~\cite{TenenbaumT-11}.

In the recent paper \cite{NakicTTV-control-prep}, the result of~\cite{TenenbaumT-11} is generalized to non-negative self-adjoint operators (regardless of the spectral type) with explicit dependence on the model parameter.
This unifies advantages of all the control cost bounds mentioned above, at least for heat flow control problems.

\begin{theorem}[\cite{NakicTTV-control-prep}]
 \label{thm:NTTV_control_abstract}
 Let $A$ be a non-negative, self-adjoint operator in a Hilbert space $\cH$ and $B \in \cL(\cU, \cH_\beta)$ for some $\beta \leq 0$, where $\cU$ is a Hilbert space and $\cH_\beta$ is defined as in Theorem~\ref{thm:Tenenbaum-Tucsnac}.
 Assume that there are $d_0 > 0$, $d_1 \geq 0$ and $s \in (0,1)$ such that for all $\lambda > 0$ we have the spectral inequality~\eqref{eq:abstract_spectral_inequality} with $C_\ur(\lambda) = d_0 \euler^{d_1 \lambda^s}$.
 Then for all $T > 0$, we have
 \begin{equation*} 
   \lVert \euler^{- A T} u_0 \rVert_\cH^2
   \leq
   C_\obs^2
   \int_0^T
   \lVert B^\ast \euler^{- A t} u_0 \rVert_\cU^2
   \drm t
 \end{equation*}
 where
 \[
  C_\obs^2
  =
  \frac{C_1 d_0}{T} K^{C_2}
  \exp
  \left(
    C_3
    \left(
    \frac{d_1+(-\beta)^{C_4}}{T^s}
    \right)^{\frac{1}{1 - s}}
  \right)
  \quad
  \text{with}
  \quad
  K
  =
  2 d_0 \euler^{-\beta} \lVert B \rVert_{\cL(\cU,\cH_\beta)} + 1.
 \]
 Here $C_1$, $C_2$, $C_3$, and $C_4$ are constants which depend only on $s$.
\end{theorem}

\section{Null-controllability of the heat and Schr\"odinger semigroups}
In the previous Section~\ref{sec:uncertainty-control}, we have seen how uncertainty relations, respectively spectral inequalities, lead to null-con\-trollabi\-lity of abstract systems.
In particular, Theorem~\ref{thm:NTTV_control_abstract} provides a very explicit estimate on the resulting control cost.
We now combine this abstract result with the results of Sections~\ref{sec:UCP_Logvinenko-Sereda} and~\ref{sec:Carleman} to deduce null-controllability
of the heat equation on cubes and on $\RR^d$ with so-called interior control and provide explicit estimates on the control cost.
In particular, the cost will be explicitly given in terms of parameters which describe the geometry of the control set.

We start by examining the classical heat equation. Recall from Section~\ref{sec:Carleman} that $\Lambda_L = (-L/2, L/2)^d \subset \RR^d$ for $L > 0$.
Let $\Omega \in \{ \Lambda_L, \RR^d \}$.
If $\Omega = \RR^d$, then $\Delta$ denotes the self-adjoint Laplacian in $L^2(\RR^d)$.
If $\Omega = \Lambda_L$, then $\Delta$ denotes the self-adjoint Laplacian in $L^2(\Lambda_L)$ with Dirichlet, Neumann or periodic boundary conditions.
Given a measurable $S \subset \RR^d$, the \emph{controlled heat equation} in time $[0,T]$ with control operator $B=\chi_{S\cap\Omega}$ (this choice is also called \emph{interior control}) is
\begin{equation}
  \label{eq:heat_equation}
  \frac{\partial}{\partial t} u - \Delta u = \chi_{S \cap \Omega} f,
  \quad
  u, f \in L^2((0,T) \times \Omega),\quad
  u(0, \cdot) = u_0 \in L^2(\Omega).
\end{equation}
Note that by the above convention, the boundary conditions are fixed by the choice of the self-adjoint Laplacian.
If $\Omega = \RR^d$, the system~\eqref{eq:heat_equation} is null-controllable if and only if $S$ is a thick set, see \cite{EgidiV-18,WangWZZ}.
If $\Omega = \Lambda_L$, the system~\eqref{eq:heat_equation} is null-controllable if and only if $\lvert \Lambda_L \cap S \rvert > 0$, see \cite{ApraizEWZ-14}.
Furthermore, in~\cite{EgidiV-18}, combining the spectral inequalities from Corollaries~\ref{cor:spectral-inequality-full-space} and \ref{cor:cubes-spectral} with
the technique by \cite{BeauchardPS-18}, cf.~Theorem~\ref{thm:BPSM}, the following estimate on the control cost is provided:

\begin{theorem}
 \label{thm:control_cost_thick_Egidi_Veselic}
 Let $L > 0$,
 $\Omega \in \{ \Lambda_L, \RR^d \}$,
 $S \subset \RR^d$ a $(\gamma, a)$-thick set with $a = (a_1, \dots, a_d)$ and $\gamma > 0$.
 If $\Omega = \Lambda_L$, we assume that $0 < a_j \le L$ for all $j = 1, \dots, d$ .
 Then, for every $T > 0$, the system~\eqref{eq:heat_equation} is null-controllable in time $T$ with cost satisfying
 \begin{equation}
 \label{eq:control_cost_thick}
  C_T
  \leq
  C_1^{1/2} \exp \left( \frac{C_1}{2 T} \right),
  \quad
  \text{where}
  \quad
  C_1
  =
  \left( \frac{K^d}{\gamma} \right)^{K ( d + \lVert a \rVert_1)},
 \end{equation}
where $K$ is a universal constant and $\lVert a \rVert_1= \sum_{j=1}^{d}a_j$.
\end{theorem}

As discussed in Section~\ref{sec:UCP_Logvinenko-Sereda}, the spectral inequalities used in the proof of Theorem~\ref{thm:control_cost_thick_Egidi_Veselic} have recently been extended in~\cite{Egidi} to strips, see Remark~\ref{rem:strip-spectral}.
This has led in an analogous way to the following result which, to the best of our knowledge, is the first result of this kind dealing with an unbounded domain $\Omega$ that is not the whole of $\RR^d$.

\begin{theorem}[\cite{Egidi}]
 \label{thm:control_cost_thick_Egidi}
 Let $L > 0$,  $\Omega =(-L/2,L/2)^{d-1}\times \RR $,
 $S \subset \RR^d$ a $(\gamma, a)$-thick set with  $\gamma > 0$.
 and $0 < a_j \le L$ for all $j = 1, \dots, d-1$.
 Then, for every $T > 0$, the system~\eqref{eq:heat_equation} with Dirichlet or Neumann boundary conditions is null-controllable in time $T$ with cost satisfying the bound  \eqref{eq:control_cost_thick}.
\end{theorem}
Here, thickness of $S$ is again a necessary requirement for null-controllability (where obviously $S$ can be arbitrarily modified outside $\Omega$).
We refer to \cite{Egidi} for more details.

In light of the discussion made in the previous section, the bound in Theorem~\ref{thm:control_cost_thick_Egidi_Veselic} (and, of course, Theorem~\ref{thm:control_cost_thick_Egidi}) can be strengthened if Theorem~\ref{thm:BPSM} in the last step of the proof is replaced by Theorem \ref{thm:NTTV_control_abstract}. For $\Omega \in \{ \Lambda_L, \RR^d \}$, this has been performed in \cite{NakicTTV-control-prep, Taeufer-18}:
\begin{theorem}
 \label{thm:control_cost_thick_2}
 Let $L > 0$,
 $\Omega \in \{ \Lambda_L, \RR^d \}$,
 $S \subset \RR^d$ a $(\gamma, a)$-thick set with $a = (a_1, \dots, a_d)$ and $\gamma > 0$.
 If $\Omega = \Lambda_L$, we assume that $0 < a_j \le L$ for all $j = 1, \dots, d$ .
 Then, for every $T > 0$, the system~\eqref{eq:heat_equation} is null-controllable in time $T$ with cost satisfying
 \begin{equation}
 \label{eq:control_cost_thick_2}
 C_T
 \leq
 \frac{D_1}{\gamma^{D_2} \sqrt{T}} \exp \left( \frac{D_3 \lVert a \rVert_1^2 \ln^2( D_4 \gamma )}{T} \right).
 \end{equation}
 where $D_1$ to $D_4$ are constants which depend only on the dimension.
\end{theorem}

\begin{proof}
 By Corollaries~\ref{cor:spectral-inequality-full-space} and~\ref{cor:cubes-spectral} we have the spectral inequality
 \[
  \forall E \geq 0, u \in \ran \chi{(-\infty, E]}(- \Delta)
  \colon
  \quad
  \lVert u \rVert_{L^2(\Omega)}^2
  \leq
  d_0 \euler^{d_1 \sqrt{E}}
  \lVert \chi_{S \cap \Omega} u \rVert_{L^2(\Omega)}^2
 \]
 with
 \[
  d_0
  =
  \left( \frac{ N_1}{\gamma} \right)^{N_2}
  \quad
  \text{and}
  \quad
  d_1 = N_3 \lVert a \rVert_1 \ln \left( \frac{N_1}{\gamma} \right),
 \]
 where $N_1$, $N_2$, and $N_3$ are constants, depending only on the dimension.
 Theorem~\ref{thm:NTTV_control_abstract} together with the equivalence between null-controllability and final-state-observability, and the absorption of all universal constants into $D_1$ to $D_4$ yields the result.
\end{proof}

\begin{remark}
 \label{rem:optimality_T}
 In order to discuss the bound~\eqref{eq:control_cost_thick_2}, let us first compare it to \emph{lower bounds} on the control cost.
 For the controlled heat equation \eqref{eq:heat_equation} with open $S$ it is known that the control cost grows at least proportional to $\exp(C/T)$ as $T$ tends to zero unless $S = \Omega$, see e.g.\ \cite{FernandezZ-00,Miller-04}.
 Thus, the $T$-dependence \eqref{eq:control_cost_thick_2} is optimal in the small time regime.

 On the other hand, the $T^{-1/2}$ term will dominate for large $T$.
 This is also optimal.
 One way to see this is to study the ODE system
 \begin{equation*} 
  \begin{cases}
   y'(t) = C f(t),
   \quad
   & y,f \in L^2((0,T), \CC),\\
   y(0) = y_0
   &\in \CC\\
  \end{cases}
 \end{equation*}
 the control cost of which can be explicitly computed and is $C/\sqrt{T}$ in time $T$ for every $T > 0$.
 This also shows that the minimal possible lower bound on the control cost in time $T$ of abstract controlled systems as in~\eqref{eq:abstract_control_system} is of order $T^{-1/2}$.
 This argument can be slightly generalized to show that this lower bound holds in fact for \emph{all} systems of the form~\eqref{eq:abstract_control_system}, see~\cite{NakicTTV-control-prep} for details.
 We conclude that the control cost in time $T$ is lower bounded by $C/\sqrt{T}$ for all $T > 0$ for some constant $C$.
 \par
 An interesting limit is the \emph{homogenization limit} of the control set where the parameter $a$ tends to zero while the parameter $\gamma$ remains constant.
 This corresponds to requiring an equidistribution on finer and finer scales $a$ while keeping the overall density $\gamma$ constant.
 We see that the exponential term, which is characteristic for the heat equation with control operator $B = \chi_{S \cap \Omega}$ where $S \subset \Omega, S \neq \Omega$, is annihilated.
 On the other hand the $1/\sqrt{T}$ factor, which is universal in the class of abstract linear control systems, remains unaffected.
 This limit can be interpreted as the control cost of the system with \emph{weighted full control}, i.e. where $\chi_S$ has been replaced by $c_\gamma \chi_{\Omega}$
with a $\gamma$-dependent constant $c_\gamma \in (0,1]$.
\end{remark}

Now we study the heat equation with \emph{non-negative} potential or homogeneous source term.
Instead of considering thick control sets $S \subset \RR^d$, we will restrict our attention to a special geometric setting, namely to equidistributed unions of $\delta$-balls.
Recall the notation from Section~\ref{sec:Carleman}: If $G > 0$, $\delta \in (0, G/2)$, and $Z$ is a $(G, \delta)$-equidistributed sequence then
\[
\equiset = \bigcup_{j \in (G\ZZ)^d} B(z_j , \delta).
\]
Let $L \geq G$ and $\Omega \in \{ \Lambda_L, \RR^d \}$.
For a non-negative $V \in L^\infty(\Omega)$ the \emph{controlled heat equation with potential $V$} in time $[0,T]$ with interior control in $\equiset \cap \Omega$ is
\begin{equation}
 \label{eq:heat_equation_with_potential}
  \frac{\partial}{\partial t} u - \Delta u + V u = \chi_{\equiset \cap \Omega} f,
  \quad
  u, f \in L^2((0,T) \times \Omega),\quad
  u(0, \cdot) = u_0 \in L^2(\Omega).
\end{equation}
In~\cite{NakicTTV-18}, Theorem \ref{thm:NakicTTV} and Miller's Theorem \ref{thm:Miller}
were combined to prove:

\begin{theorem}\label{thm:NTTV_control}
There exists $T' > 0$, depending on $G$, $\delta$, and $\lVert V \rVert_\infty$ such that for all $T \leq T'$, the system~\eqref{eq:heat_equation_with_potential} is null-controllable in time $T$ with cost $C_T$ satisfying
 \[
  C_T
  \leq
  2
  \left( \frac{G}{\delta} \right)^{K (1 + G^{4/3} \lVert V \rVert_\infty^{2/3})}
  \exp
  \left(
   \lVert V \rVert_\infty
   +
   \frac{\ln^2( \delta/G)
   \left(
   K G + 4 / \ln(2)
   \right)^2
   }{T}
  \right)
 \]
with a dimension-dependent $K$.
\end{theorem}
Again we can improve this bound by replacing Theorem \ref{thm:Miller} with a more suitable estimate. Furthermore,
certain unbounded domains can be treated as well.
More precisely, combining Theorems~\ref{thm:NTTVS} and~\ref{thm:NTTV_control_abstract}, we obtain
analogously to Theorem~\ref{thm:control_cost_thick_2} the following result.
\begin{theorem}
 \label{thm:control_cost_equidistributed}
 Let $G > 0$, $0 < \delta < G/2$, $Z$ a $(G,\delta)$-equidistributed sequence, $L \geq G$, and $\Omega \in \{ \Lambda_L, \RR^d \}$.
 Then, for every $T > 0$, the system~\eqref{eq:heat_equation_with_potential} is null-controllable in time $T$ with cost satisfying
 \begin{equation}
 \label{eq:control_cost_equidistributed}
  C_T
 \leq
 \frac{D_1}{\sqrt{T}}
 \left( \frac{G}{\delta} \right)^{D_2 (1 + G^{4/3} \lVert V \rVert_\infty^{2/3})}
 \exp \left( \frac{D_3 G^2 \ln^2( \delta/G)}{T} \right).
 \end{equation}
 where $D_1, D_2, D_3$ are constants which depend only on the dimension.
\end{theorem}

Theorem~\ref{thm:control_cost_equidistributed} improves upon Theorem~\ref{thm:NTTV_control} since it allows for all times $T > 0$ and since the argument of the exponential term is now of order $G^2$ as $G \to 0$, which is optimal.

The difference between Theorems~\ref{thm:control_cost_thick_2} and~\ref{thm:control_cost_equidistributed} is that Theorem~\ref{thm:control_cost_thick_2} allows for more general control sets, while Theorem~\ref{thm:control_cost_equidistributed} treats Schr\"odinger operators with non-negative potential instead of the pure Laplacian.

\begin{remark}
  By the same arguments as in Remark~\ref{rem:optimality_T}, we see that the asymptotic $T$-dependence in Theorem~\ref{thm:control_cost_equidistributed} is optimal.
  Homogenization of the control set now corresponds to $G, \delta \to 0$ with $\delta/G = \rho$ for some $\rho \in (0, 1/2)$.
  In the limit, the upper bound in~\eqref{eq:control_cost_equidistributed} tends to
  \[
   \frac{D_1}{\sqrt{T}}
   \rho^{D_2}.
  \]
  We see that homogenization not only annihilates the term $\exp(C/T)$ which is characteristic for the heat equation, but also the influence of a non-negative potential $V$ on the control cost estimate disappears.
  \par
  Furthermore, the dependence of the exponential term on the parameter $G$ in~\eqref{eq:control_cost_equidistributed} is optimal.
  This can best be seen in the special case $V = 0$ by comparing it to a lower bound on the control cost in terms of the geometry deduced in~\cite{Miller-04}.
  In fact, for the heat equation on smooth, connected manifolds $\Omega$ with control operator $B = \chi_S$ for an open $S \subset \Omega$ it is proved in~\cite{Miller-04} that the control cost $C_T$ in time $T$ satisfies
 \begin{equation}
  \label{eq:lower_bound_geometry_Miller}
  \sup_{\overline B_\rho \subset \Omega \backslash \overline S}
  \rho^2/4
  \leq
  \liminf_{T \to 0} T \ln C_T.
 \end{equation}
 Ineq.~\eqref{eq:control_cost_equidistributed} on the other hand implies
 \begin{equation}
  \label{eq:upper_bound_geometry}
  \limsup_{T \to 0} T \ln C_T
  \leq
  D_3 G^2 \ln^2(\delta / G).
 \end{equation}
 Thus, we complement the lower bound in~\eqref{eq:lower_bound_geometry_Miller} by an upper bound.
 More precisely, for a $(G,\delta)$-equidistributed sequence $Z$, it is clear that the complement of $\overline \equiset$ (in $\Lambda_L$ or $\RR^d$, respectively) always contains a ball of radius
 \[
  \rho = \frac{1}{2} \left( \frac{G}{2} - \delta \right)
  =
  G \frac{1 - 2 \delta/G}{4}
  \quad
  \text{whence}
  \quad
  G\frac{1 - 2 \delta/G}{4}
  \leq
  \sup_{B_\rho \subset \Omega \backslash \equiset} \rho.
 \]
 Combining this with~\eqref{eq:lower_bound_geometry_Miller} and~\eqref{eq:upper_bound_geometry}, we find
 \[
  G^2\frac{(1 - 2 \delta/G)^2}{64}
\leq
  \sup_{B_\rho \subset \Omega \backslash \equiset}
  \rho^2/4
  \leq
  \liminf_{T \to 0} T \ln C_T
  \leq
  \limsup_{T \to 0} T \ln C_T
  \leq
 D_3 G^2 \ln^2( \delta/G) .
 \]
 If we perform  the limit $G \to 0$ or $G \to \infty$, respectively, while keeping $\delta/G$ constant, this reasoning shows that the factor $G^2$ in the exponential term in~\eqref{eq:control_cost_equidistributed} is optimal.
 \end{remark}

\begin{remark}
 So far, we only used the fact that $V \geq 0$.
 If however, we have $V \geq \kappa > 0$, then the control cost should decay proportional to $\exp(- \kappa T)$ at large times.
 This can be seen by modifying the construction of the null-control, see~\cite{Taeufer-18, NakicTTV-control-prep}

 Conversely, if we only have $V \in L^\infty$, but $\inf V < 0$, then the situation might become even more interesting.
 In fact, the relevant quantity is $\min \sigma (- \Delta + V)$.
 If $\min \sigma (- \Delta + V) < 0$, then the semigroup $\exp((\Delta - V) t)$ will be non-contractive and the control cost will be bounded away from zero uniformly for all times $T > 0$.
 This situation can also be studied by an appropriate generalization of the above arguments, see~\cite{Taeufer-18, NakicTTV-control-prep}.
\end{remark}

\begin{remark}
 One can also study the \emph{fractional heat equation} for $\theta \in (1/2, \infty)$:
\begin{equation}
  \label{eq:fractional_heat_equation}
  \frac{\partial}{\partial t} u + (- \Delta)^\theta u = \chi_{S \cap \Omega} f,
  \quad
  u, f \in L^2((0,T) \times \Omega),\quad
  u(0, \cdot) = u_0 \in L^2(\Omega)
\end{equation}
 and deduce an estimate on the control cost.
 Here, again $\Omega \in \{ \Lambda_L, \RR^d \}$, and $S$ is a $(\gamma,a)$-thick set such that $0 < a_j \le L$ for all $j = 1, \dots, d$ in case $\Omega = \Lambda_L$.
 It is known that the fractional heat equation on one-dimensional intervals is null-controllable if and only if $\theta > 1/2$, see~\cite{MicuZ-06}.
 In order to deduce a control cost estimate, it suffices to deduce an uncertainty relation for the operator $(- \Delta)^\theta$.
 For that purpose, we estimate using the transformation formula for spectral measures, cf.~\cite[Prop.~4.24]{Schmuedgen-12}, and the uncertainty relation for the pure Laplacian in Corollaries~\ref{cor:spectral-inequality-full-space} and~\ref{cor:cubes-spectral}
 \begin{equation}
  \begin{aligned}
    \label{eq:uncertainty_relation_fractional_Laplacian}
    \lVert \chi_{(-\infty,\lambda]}(- \Delta)^\theta)  u \rVert_{L^2(\Omega)}^2
    &=
    \lVert \chi_{(-\infty,\lambda^{1/\theta}]}(- \Delta) u \rVert_{L^2(\Omega)}^2\\
    &\leq
    d_0 \euler^{d_1 \lambda^{1/(2 \theta)}}
    \lVert \chi_S \cdot \chi_{(-\infty,\lambda^{1/\theta}]}(- \Delta) u \rVert_{L^2(\Omega)}^2\\
    &=
    d_0 \euler^{d_1 \lambda^{1/(2 \theta)}}
    \lVert \chi_S \cdot \chi_{(-\infty,\lambda]}((- \Delta)^\theta) u \rVert_{L^2(\Omega)}^2
 \end{aligned}
 \end{equation}
 for all $\lambda \geq 0$ and all $u \in L^2(\Omega)$ where
 \[
  d_0
  =
  \left( \frac{ N_1}{\gamma} \right)^{N_2}
  \quad
  \text{and}
  \quad
  d_1 = N_3 \lVert a \rVert_1 \ln \left( \frac{N_1}{\gamma} \right),
 \]
 with constants $N_1$, $N_2$, and $N_3$, depending only on the dimension.
 Combining~\eqref{eq:uncertainty_relation_fractional_Laplacian} and Theorem~\ref{thm:NTTV_control_abstract}, we obtain the following result.
\end{remark}

\begin{corollary}
Let $\theta \in (1/2, \infty)$, $\Omega \in \{ \Lambda_L, \RR^d \}$, and $S$ be a $(\gamma,a)$-thick set such that,  in case $\Omega = \Lambda_L$, $0 < a_j \le L$ for all $j = 1, \dots, d$.
Then the system~\eqref{eq:fractional_heat_equation} is null-controllable in any time $T > 0$ with cost satisfying
 \begin{equation*} 
  C_T
  \leq
  \frac{D_1}{\gamma^{D_2} \sqrt{T}} \exp \left( \frac{D_3 \left( \lVert a \rVert_1 \ln( D_4 / \gamma  ) \right)^{\frac{2 \theta}{2 \theta - 1}}}{T^{\frac{1}{2 \theta - 1}}} \right)
 \end{equation*}
 for constants $D_1, \dots, D_4$, depending only on $\theta>1/2$ and on the dimension.
\end{corollary}

\section{Convergence of solutions along exhausting cubes}

In this section we review certain approximation results which have been indicated in \cite{EgidiV-18} and spelled out with proofs in \cite{SeelmannV}.
They describe how controllability problems on unbounded domains can be approximated by corresponding problems on a sequence of bounded domains. Since these results apply to a larger class of Schr\"odinger operators than discussed so far, we will
introduce them first.

Let $\Omega\subset\RR^d$ be an open set and $\Lambda_L:=(-L/2,L/2)^d$ with $L>0$ as before.
Let $V\colon\RR^d\to\RR$ be a potential such that $V_+:=\max(V,0)\in L_{\mathrm{loc}}^1(\Omega)$ and $V_-:=\max(-V,0)$ is in the Kato
class; see, e.g., \cite[Section 1.2]{CyconFKS-87} for a discussion of the Kato class in $\RR^d$.
Under these hypotheses, one can define the Dirichlet Schr\"odinger operators $H_\Omega$ and $H_L=H_{\Omega\cap\Lambda_L}$
as lower semi-bounded self-adjoint operators on $L^2(\Omega)$ and $L^2(\Omega\cap\Lambda_L)$, respectively, associated with the
differential expression $-\Delta + V$ via their quadratic forms, with form core $C_c^\infty(\Omega)$ and
$C_c^\infty(\Omega\cap\Lambda_L)$, respectively.
For details of this construction we refer to \cite[Section 1.2]{CyconFKS-87}, \cite[Section 2]{HundertmarkS-04}, and the references therein.
In fact, our arguments apply to Schr\"odinger operators incorporating a magnetic vector potential as well, see \cite{SeelmannV} for
details.

Since we want to compare operators defined on two different Hilbert spaces, namely $L^2(\Omega)$ and $L^2(\Omega\cap\Lambda_L)$, we
need a notion of extension.
Corresponding to the orthogonal decomposition $L^2(\Omega)=L^2(\Omega\cap\Lambda_L)\oplus L^2(\Omega\setminus\Lambda_L)$,
we identify $H_L$ with the direct sum $H_L\oplus0$ on $L^2(\Omega)$.
Consequently, the subspace $L^2(\Omega\cap\Lambda_L)\subset L^2(\Omega)$ is a reducing subspace for the self-adjoint operator
$H_L$ on $L^2(\Omega)$. Hence, the exponential $\euler^{-tH_L}=\euler^{-tH_L}\oplus I$ for all $t\ge 0$ decomposes as well; see,
e.g., \cite[Definition~1.8]{Schmuedgen-12} and \cite[Satz~8.23]{Weidmann-00}. In particular, $\euler^{-tH_L}$ is a bounded
self-adjoint operator on $L^2(\Omega)$, and $\euler^{-tH_L}f=0$ on $\Omega\setminus\Lambda_L$ for all $f\in L^2(\Omega\cap\Lambda_L)$.

\subsection{Approximation of semigroups based on an exhaustion of the domain}

An important tool in what follows is an approximation result for Schr\"odinger semigroups. It applies to a sequence of
semigroups, all of the same type, but defined on different domains.

\begin{lemma}[\cite{SeelmannV}] 
 Let $R>0$, $u_0\in L^2(\Omega\cap\Lambda_R)\subset L^2(\Omega)$, and $t>0$. Then, there exists a constant $C=C(t,d,V_-)>0$ such
 that for every $L\ge 2R$ one has
 \begin{equation*}
  \norm{(\euler^{-tH_\Omega}-\euler^{-tH_L})u_0}_{L^2(\Omega)}^2 \le C\exp\bigl(-\frac{L^2}{32t}\bigr)\norm{u_0}_{L^2(\Omega)}^2.
 \end{equation*}
\end{lemma}

The lemma implies that for every $t>0$ the exponential $\euler^{-tH_L}$ converges strongly to $\euler^{-tH_\Omega}$ as $ L\to \infty$.
Moreover, it exhibits a very explicit error bound if the support of the function $u_0$ is located inside some cube. However,
for what we present here the qualitative statement on strong convergence will be all what we will use.

\subsection{Continuous dependence on inhomogeneity}
%
In the applications we have in mind, the above approximation estimate for a sequence of semigroups needs to be complemented by an
approximation result with respect to change of the right-hand side of the partial differential equation and truncation of the
initial datum. This is presented next in a more general framework.

Let $\cH$ and $\cU$ be Hilbert spaces, and let $T>0$. Recall (cf.~Section~\ref{sec:uncertainty-control} above) that given a lower semi-bounded
self-adjoint operator $A$ on $\cH$, a bounded operator $B\colon \cU \to \cH$, $u_0\in\cH$, and $f\in L^2((0,T),\cU)$, the continuous function
$u\colon[0,T]\to\cH$ with
\begin{equation*}
 u(t) = \euler^{-tA}u_0 + \int_0^t \euler^{-(t-s)A}Bf(s) \drm s
\end{equation*}
is called the \emph{mild solution} to the abstract Cauchy problem
\begin{equation*}
 \frac{\partial}{\partial t} u(t) + Au(t) = Bf(t) \quad\text{ for }\quad 0<t<T,\quad u(0)=u_0.
\end{equation*}

\begin{lemma}[\cite{SeelmannV}]\label{lem:weakStrongConv}
 Let $A,A_n$, $n\in\NN$, be lower semi-bounded self-adjoint operators on the Hilbert space $\cH$ with a common lower bound
 $a\in\RR$. Assume that $(\euler^{-tA_n})_n$ converges strongly to $\euler^{-tA}$ for all $t>0$. Let $B,B_n$, $n\in\NN$, be bounded
 operators from $\cU$ to $\cH$  such that $(B_n)_n$ and $(B_n^*)_n$ converge strongly to $B$ and $B^*$, respectively.
 Moreover, let $(u_{0,n})_n$ be a sequence in $\cH$
 converging in norm to some $u_0\in\cH$. Let $f,f_n\in L^2((0,T),\cU)$, $n\in\NN$. Denote by $u$ and $u_n$, $n\in\NN$, the mild
 solutions to the abstract Cauchy problems
 \begin{equation*}
  \frac{\partial}{\partial t} u(t) + A u(t) = Bf(t) \quad\text{ for }\quad 0<t<T,\quad u(0)=u_0,
 \end{equation*}
 and
 \begin{equation*}
  \frac{\partial}{\partial t}   u_n(t) + A_n u_n(t) = B_nf_n(t) \quad\text{ for }\quad 0<t<T,\quad u_n(0)=u_{0,n},
 \end{equation*}
 respectively.
 \begin{enumerate}[(a)]
  \item If $(f_n)_n$ converges to $f$ in $L^2((0,T),\cU)$, then $(u_n(t))_n$ converges to $u(t)$ in $\cH$ for all $t\in(0,T]$.
        Moreover, $(u_n)_n$ converges to $u$ in $L^2((0,T),\cU)$.
  \item If $(f_n)_n$ converges to $f$ weakly in $L^2((0,T),\cU)$, then $(u_n(t))_n$ converges to $u(t)$ weakly in $\cH$ for all
        $t\in(0,T]$. Moreover, the sequence $(u_n)_n$ converges to $u$ weakly in $L^2((0,T),\cH)$.
 \end{enumerate}
\end{lemma}

If the sequence $(f_n)_n$ consists of null-controls as in Definition~\ref{def:null_controllable}, then this property is inherited by the limit $f$, more precisely:
\begin{corollary}
 If in the situation of Lemma \ref{lem:weakStrongConv} the sequence $(f_n)_n$ converges to $f$ weakly in $L^2((0,T),\cU)$ and
 for every $n$ one has $u_n(T)=0$, then also $u(T)=0$.
\end{corollary}

\subsection{Construction of controls via exhaustion of the domain}
In certain situations it may be easier to infer (or is already known) that a certain variant of the heat equation exhibits a null-control
provided the domain of the problem is bounded. With the operators $H_L$ and $H_\Omega$ introduced above we present a criterion how
one can infer the existence of a null-control of the corresponding problem on an unbounded domain.

\begin{theorem}[\cite{SeelmannV}]\label{thm:main}
 Let $S\subset\RR^d$ be measurable, $\tilde u\in L^2(\Omega)$, and $(L_n)_n$ a sequence in $(0,\infty)$ with $L_n\nearrow\infty$ as
 $n\to\infty$. Let $f_n\in L^2((0,T),L^2(\Omega\cap\Lambda_{L_n}\cap S))$ for each $n\in\NN$ be a null-control for the initial
 value problem
 \begin{equation}\label{eq:Schroedinger-control-Gamma_n}
 \frac{\partial}{\partial t} u(t) + H_{L_n} u(t) = \chi_{\Omega \cap \Lambda_{L_n}\cap S}f_n(t)\quad\text{ for }\quad 0<t<T,\quad u(0)=
 {\chi_{\Omega \cap \Lambda_{L_n}}\tilde u},
 \end{equation}
 and $u_n$ the corresponding mild solution.

 Suppose that $(f_n)_n$ converges weakly in $L^2((0,T),L^2(\Omega))$ to some function $f$. Then, $f$ is a null-control for
 \begin{equation}\label{eq:Schroedinger-control-Gamma}
 \frac{\partial}{\partial t} u(t) + H_{\Omega} u(t) = \chi_{\Omega\cap S}f(t)\quad\text{ for }\quad 0<t<T,\quad u(0)=\tilde u,
 \end{equation}
 and the corresponding mild solution is the weak limit of $(u_n)_n$ in $L^2((0,T),L^2(\Omega))$.
\end{theorem}

The above theorem is based on Lemma~\ref{lem:weakStrongConv} in the situation $\cU=\cH=L^2(\Omega)$ with $A=H_\Omega$,
$A_n=H_{L_n}$, $B=\chi_{\Omega\cap S}$, and $B_n=\chi_{\Omega\cap\Lambda_{L_n}\cap S}$.
In this case, due to the discussion before Definition~\ref{def:optimal-feedback},
the null-controls for~\eqref{eq:Schroedinger-control-Gamma_n} and~\eqref{eq:Schroedinger-control-Gamma} can indeed be assumed to be
supported in $\Omega\cap S$ and $\Omega\cap\Lambda_{L_n}\cap S$, respectively.

Note that if the null-controls $f_n$ in Theorem~\ref{thm:main} are uniformly bounded, that is,
\begin{equation}\label{eq:mainCrit}
 \norm{f_n}_{L^2((0,T),L^2(\Omega\cap\Lambda_{L_n}\cap S))} \le c \quad\text{ for all }\quad n\in\NN
\end{equation}
for some constant $c>0$, then $(f_n)_n$ has a weakly convergent subsequence with limit in $L^2((0,T),L^2(\Omega\cap S))$.
Theorem~\ref{thm:main} can then be applied to every such weakly convergent subsequence, and the corresponding weak limit $f$ of the
subsequence of $(f_n)_n$ automatically satisfies the bound
\begin{equation}\label{eq:limitBound}
 \norm{f}_{L^2((0,T),L^2(\Omega\cap S))}\le c.
\end{equation}

This leads to the following corollary to Theorem \ref{thm:main}.

\begin{corollary} 
 Let $S\subset\RR^d$ be measurable, $\tilde u\in L^2(\Omega)$, and $(L_n)_n$ a sequence in $(0,\infty)$ with $L_n\nearrow\infty$ as
 $n\to\infty$. Let $f_n\in L^2((0,T),L^2(\Omega\cap\Lambda_{L_n}\cap S))$ for each $n\in\NN$ be a null-control for the initial
 value problem \eqref{eq:Schroedinger-control-Gamma_n}, and let $u_n$ be the corresponding mild solution.

 Assume that there is a constant $c\in \RR $ such that \eqref{eq:mainCrit} holds. Then there exists a subsequence of $(f_n)_n$ which
 converges weakly to a null-control $f\in L^2((0,T),L^2(\Omega\cap S))$ for \eqref{eq:Schroedinger-control-Gamma}, satisfying
 \eqref{eq:limitBound} as well. The mild solution $u$ associated to (any such weak accumulation point) $f$ is the weak limit of the
 corresponding subsequence of $(u_n)_n$ in $L^2((0,T),L^2(\Omega))$.
\end{corollary}

As discussed in previous sections, the control cost estimate \eqref{eq:mainCrit} can be inferred by a final state observability
estimate. Consequently, a scale-free uncertainty principle or spectral inequality, as formulated in Theorem \ref{thm:log-ser-cubes}
or Theorem \ref{thm:NakicTTV}, leads not only to control cost estimates on a sequence of bounded cubes $\Lambda_L$ but also to the
limiting domain $\Omega=\RR^d$. This means that results like Theorem~\ref{thm:control_cost_thick_Egidi_Veselic} or
Theorem~\ref{thm:control_cost_equidistributed} (for $\Omega=\RR^d$) could be obtained by a (partially) alternative method, where
one performs hard analysis for partial differential equations only on bounded domains and then invokes operator theoretic methods
to lift the results to unbounded domains. For details see \cite{SeelmannV}.

\vspace{5mm}

\noindent\textbf{Acknowledgments}\\
The initial phase of this research was supported by travel grants within the binational Croatian-German PPP-Project \emph{The cost of controlling the heat flow in a multiscale setting}.
M.~E.~and I.~V.~were supported in part by the Deutsche Forschungsgemeinschaft under the project grant Ve 253/7-1 \emph{Multiscale Version of the Logvinenko--Sereda Theorem}.
I.~N.~was supported in part by the Croatian Science Foundation under the projects 9345 and IP-2016-06-2468.


\begin{thebibliography}{WWZZ19}

\bibitem[AEWZ14]{ApraizEWZ-14}
J.~Apraiz, L.~Escauriaza, G.~Wang, and C.~Zhang, \emph{Observability
  inequalities and measurable sets}, J. Eur. Math. Soc. \textbf{16} (2014),
  no.~11, 2433--2475.

\bibitem[And14]{Andersen-14}
{N.~B.} Andersen, \emph{Entire ${L}^p$-functions of exponential type}, Expo.
  Math. \textbf{32} (2014), no.~3, 199--220.

\bibitem[AW15]{AizenmanW-15}
M.~Aizenman and S.~Warzel, \emph{Random operators. {D}isorder effects on
  quantum spectra and dynamics}, Graduate Studies in Mathematics, vol. 168,
  American Mathematical Society, Providence, 2015.

\bibitem[BK05]{BourgainK-05}
J.~Bourgain and {C.~E.} Kenig, \emph{On localization in the continuous
  {A}nderson-{B}ernoulli model in higher dimension}, Invent. Math. \textbf{161}
  (2005), no.~2, 389--426.

\bibitem[BP18]{BeauchardPS-18}
K.~Beauchard and K.~{Pravda-Starov}, \emph{Null-controllability of hypoelliptic
  quadratic differential equations}, J. \'Ec. polytech. Math. \textbf{5}
  (2018), 1--43.

\bibitem[CFKS87]{CyconFKS-87}
{H.~L.} Cycon, {R.~G.} Froese, W.~Kirsch, and B.~Simon, \emph{Schr\"odinger
  operators with application to quantum mechanics and global geometry}, Texts
  and Monographs in Physics, Springer, Berlin, 1987.

\bibitem[CHK03]{CombesHK-03}
{J.-M.} Combes, {P.~D.} Hislop, and F.~Klopp, \emph{H\"older continuity of the
  integrated density of states for some random operators at all energies}, Int.
  Math. Res. Not. \textbf{4} (2003), 179--209.

\bibitem[CHK07]{CombesHK-07}
\bysame, \emph{An optimal {W}egner estimate and its application to the global
  continuity of the integrated density of states for random {S}chr\"odinger
  operators}, Duke Math. J. \textbf{140} (2007), no.~3, 469--498.

\bibitem[CP78]{CurtainP78}
{R.~F.} Curtain and {A.~J.} Pritchard, \emph{Infinite dimensional linear
  systems theory}, Springer, Berlin, 1978.

\bibitem[Dou66]{Douglas-66}
{R. G.} Douglas, \emph{On majorization, factorization, and range inclusion of
  operators on {H}ilbert space}, Proc. Amer. Math. Soc. \textbf{2} (1966),
  no.~17, 413--415.

\bibitem[DR77]{DoleckiR-77}
S.~Dolecki and {D.~L.} Russell, \emph{A general theory of observation and
  control}, SIAM J. Control Optim. \textbf{2} (1977), no.~15, 185--220.

\bibitem[Egi]{Egidi}
M.~Egidi, \emph{On null-controllability of the heat equation on infinite strips
  and control cost estimate}, to appear in Math. Nachr., arXiv:1809.10942
  [math.AP].

\bibitem[EN99]{EngelN-99}
{K.-J.} Engel and R.~Nagel, \emph{One-parameter semigroups for linear evolution
  equations}, Graduate Texts in Mathematics, vol. 194, Springer, New York,
  1999.

\bibitem[EV]{EgidiV-16-arxiv}
M.~Egidi and I.~Veseli\'c, \emph{Scale-free unique continuation estimates and
  {L}ogvi\-nenko-{S}ereda {T}heorems on the torus}, arXiv:1609.07020 [math.CA].

\bibitem[EV18]{EgidiV-18}
\bysame, \emph{Sharp geometric condition for null-controllability of the heat
  equation on {$\mathbb{R}^d$} and consistent estimates on the control cost},
  Arch. Math. \textbf{111} (2018), no.~1, 85--99.

\bibitem[FI96]{FursikovI-96}
{A.~V.} Fursikov and {O.~Y.} Imanuvilov, \emph{Controllability of evolution
  equations}, Suhak kang\v{u}irok, vol.~34, Seoul National University, Seoul,
  1996.

\bibitem[FZ00]{FernandezZ-00}
E.~{Fern{\'a}ndez-Cara} and E.~Zuazua, \emph{The cost of approximate
  controllability for heat equations: the linear case}, Adv. Differential
  Equations \textbf{5} (2000), no.~4--6, 465--514.

\bibitem[Ger08]{Germinet-08}
F.~Germinet, \emph{Recent advances about localization in continuum random
  {S}chr\"odinger operators with an extension to underlying {D}elone sets},
  Mathematical results in quantum mechanics (I.~Beltita, G.~Nenciu, and
  R.~Purice, eds.), World Scientific, Singapore, 2008, pp.~79--96.

\bibitem[GK13]{GerminetK-13b}
F.~Germinet and A.~Klein, \emph{A comprehensive proof of localization for
  continuous {A}nderson models with singular random potentials}, J. Eur. Math.
  Soc. \textbf{15} (2013), no.~1, 53--143.

\bibitem[HS04]{HundertmarkS-04}
D.~Hundertmark and B.~Simon, \emph{A diamagnetic inequality for semigroup
  differences}, J. Reine Angew. Math. \textbf{571} (2004), 107--130.

\bibitem[HV07]{HelmV-07}
M.~Helm and I.~Veseli{\'c}, \emph{Linear {W}egner estimate for alloy-type
  {S}chr\"odinger operators on metric graphs}, J. Math. Phys. \textbf{48}
  (2007), no.~9, 092107, 7.

\bibitem[JL99]{JerisonL-99}
D.~Jerison and G.~Lebeau, \emph{Nodal sets of sums of eigenfunctions}, Harmonic
  analysis and partial differential equations (M.~Christ, {C.~E.} Kenig, and
  C.~Sadosky, eds.), Chicago Lectures in Mathematics, The University of Chicago
  Press, Chicago, 1999, pp.~223--239.

\bibitem[Kac73]{Katsnelson-73}
{V.~\`E.} Kacnel'son, \emph{Equivalent norms in spaces of entire functions},
  Math. USSR Sb. \textbf{21} (1973), no.~1, 33--55.

\bibitem[Kir96]{Kirsch-96}
W.~Kirsch, \emph{{W}egner estimates and {A}nderson localization for alloy-type
  potentials}, Math. Z. \textbf{221} (1996), no.~1, 507--512.

\bibitem[Kle13]{Klein-13}
A.~Klein, \emph{Unique continuation principle for spectral projections of
  {Schr{\"o}dinger} operators and optimal {Wegner} estimates for non-ergodic
  random {Schr{\"o}dinger} operators}, Comm. Math. Phys. \textbf{323} (2013),
  no.~3, 1229--1246.

\bibitem[Kov00]{Kovrijkine-00}
O.~Kovrijkine, \emph{Some estimates of {F}ourier transforms}, Ph.D. thesis,
  California Institute of Technology, 2000.

\bibitem[Kov01]{Kovrijkine-01}
\bysame, \emph{Some results related to the {L}ogvinenko-{S}ereda theorem},
  Proc. Amer. Math. Soc. \textbf{129} (2001), no.~10, 3037--3047.

\bibitem[KSS98]{KirschSS-98a}
W.~Kirsch, P.~Stollmann, and G.~Stolz, \emph{Localization for random
  perturbations of periodic {S}chr\"odinger operators}, Random Oper. Stoch.
  Equ. \textbf{6} (1998), no.~3, 241--268.

\bibitem[KV02]{KirschV-02b}
W.~Kirsch and I.~Veseli{\'c}, \emph{Existence of the density of states for
  one-dimensional alloy-type potentials with small support}, Mathematical
  Results in Quantum Mechanics (R.~Weber, P.~Exner, and B.~Gr\'ebert, eds.),
  Contemp. Math., vol. 307, American Mathematical Society, Providence, 2002,
  pp.~171--176.

\bibitem[Lio88]{Lions-88}
{J.-L.} Lions, \emph{Contr{\^o}labilit{\'e} exacte, perturbations et
  stabilisation de syst{\`e}mes distribu{\'e}s. {T}ome 1}, Rech. Math. Appl.
  \textbf{8} (1988).

\bibitem[LL12]{RousseauL-12}
J.~{Le Rousseau} and G.~Lebeau, \emph{On {C}arleman estimates for elliptic and
  parabolic operators. {A}pplications to unique continuation and control of
  parabolic equations}, ESAIM Control Optim. Calc. Var. \textbf{18} (2012),
  no.~3, 712--747.

\bibitem[LR95]{LebeauR-95}
G.~Lebeau and L.~Robbiano, \emph{Contr{\^o}le exact de l'{\'e}quation de la
  chaleur}, Comm. Partial Differential Equations \textbf{20} (1995), no.~1--2,
  335--356.

\bibitem[LS74]{Logvinenko-Sereda-74}
{V.~N.} Logvinenko and {Ju.~F.} Sereda, \emph{Equivalent norms in spaces of
  entire functions of exponential type}, Teor. Funkci\u\i \ Funkcional. Anal. i
  Prilo\v zen. (1974), no.~20, 102--111.

\bibitem[LZ98]{LebeauZ-98}
G.~Lebeau and E.~Zuazua, \emph{Null-controllability of a system of linear
  thermoelasticity}, Arch. Ration. Mech. Anal. \textbf{141} (1998), no.~4,
  297--329.

\bibitem[Mil04]{Miller-04}
L.~Miller, \emph{Geometric bounds on the growth rate of null-controllability
  cost for the heat equation in small time}, J. Differential Equations
  \textbf{204} (2004), no.~1, 202--226.

\bibitem[Mil10]{Miller-10}
\bysame, \emph{A direct {L}ebeau-{R}obbiano strategy for the observability of
  heat-like semigroups}, Discrete Contin. Dyn. Syst. Ser. B \textbf{14} (2010),
  no.~4, 1465--1485.

\bibitem[MZ06]{MicuZ-06}
S.~Micu and E.~Zuazua, \emph{On the controllability of a fractional order
  parabolic equation}, SIAM J. Control Optim. \textbf{44} (2006), no.~6,
  1950--1972.

\bibitem[Naz93]{Nazarov-93}
{F.~L.} Nazarov, \emph{Local estimates for exponential polynomials and their
  applications to inequalities of the uncertainty principle type}, Algebra i
  Analiz \textbf{5} (1993), no.~4, 3--66.

\bibitem[NTTVa]{NakicTTV-control-prep}
I.~Naki{\'c}, M.~T{\"a}ufer, M.~Tautenhahn, and I.~Veseli{\'c}, \emph{Sharp
  estimates and homogenization of the control cost of the heat equation on
  large domains}, to appear in ESAIM Control Optim. Calc. Var., DOI:
  10.1051/cocv/2019058.

\bibitem[NTTVb]{NakicTTV-18-arxiv}
\bysame, \emph{Unique continuation and lifting of spectral band edges of
  {S}chr\"odinger operators on unbounded domains}, with an appendix by Albrecht
  Seelmann, to appear in J. Spectr. Theory, arXiv:1804.07816 [math.SP].

\bibitem[NTTV15]{NakicTTV-15}
\bysame, \emph{Scale-free uncertainty principles and {Wegner} estimates for
  random breather potentials}, C. R. Math. Acad. Sci. Paris \textbf{353}
  (2015), no.~10, 919 -- 923.

\bibitem[NTTV18]{NakicTTV-18}
\bysame, \emph{Scale-free unique continuation principle, eigenvalue lifting and
  {W}egner estimates for random {S}chr\"odinger operators}, Anal. PDE
  \textbf{11} (2018), no.~4, 1049--1081.

\bibitem[Pan61]{Panejah-61}
{B. P.} Panejah, \emph{Some theorems of {P}aley-{W}iener type}, Soviet Math.
  Dokl. \textbf{2} (1961), 533--536.

\bibitem[Pan62]{Panejah-62}
\bysame, \emph{On some problems in harmonic analysis}, Dokl. Akad. Nauk
  \textbf{142} (1962), 1026--1029.

\bibitem[RS75]{ReedS-vol2-75}
M.~Reed and B.~Simon, \emph{Methods of modern mathematical physics. {II}.
  {F}ourier analysis, self-adjointness}, Academic Press, New York, 1975.

\bibitem[Rus78]{Russell78}
{D.~L.} Russell, \emph{Controllability and stabilizability theory for linear
  partial differential equations: recent progress and open questions}, SIAM
  Rev. \textbf{20} (1978), no.~4, 639--739.

\bibitem[RV13]{RojasMolinaV-13}
C.~{Rojas-Molina} and I.~Veseli{\'c}, \emph{Scale-free unique continuation
  estimates and applications to random {S}chr\"odinger operators}, Comm. Math.
  Phys. \textbf{320} (2013), no.~1, 245--274.

\bibitem[Sch12]{Schmuedgen-12}
K.~Schm{\"u}dgen, \emph{Unbounded self-adjoint operators on {H}ilbert space},
  Graduate Texts in Mathematics, vol. 265, Springer, Dordrecht, 2012.

\bibitem[Sto01]{Stollmann-01}
P.~Stollmann, \emph{Caught by disorder: Bound states in random media}, Progress
  in Mathematical Physics, vol.~20, Birkh\"auser, 2001.

\bibitem[SV20]{SeelmannV}
A.~Seelmann and I.~Veseli\'c, \emph{Exhaustion approximation for the control
  problem of the heat or {S}chr\"odinger semigroup on unbounded domains}, Arch.
  Math. \textbf{115} (2020), no.~2, 195--213.

\bibitem[T{\"a}u18]{Taeufer-18}
M.~T{\"a}ufer, \emph{Quantitative unique continuation and applications}, Ph.D.
  thesis, Technische Universit\"at Dortmund, 2018.

\bibitem[TT11]{TenenbaumT-11}
G.~Tenenbaum and M.~Tucsnak, \emph{On the null-controllability of diffusion
  equations}, ESAIM Control Optim. Calc. Var. \textbf{17} (2011), no.~4,
  1088--1100.

\bibitem[TT17]{TaeuferT-17}
M.~T\"aufer and M.~Tautenhahn, \emph{Scale-free and quantitative unique
  continuation for infinite dimensional spectral subspaces of {S}chr{\"o}dinger
  operators}, Commun. Pure Appl. Anal. \textbf{16} (2017), no.~5, 1719--1730.

\bibitem[Tur46]{Turan-46}
P.~Tur\'an, \emph{On a theorem of {L}ittlewood}, J. Lond. Math. Soc.
  \textbf{21} (1946), no.~4, 268--275.

\bibitem[TV16]{TautenhahnV-16}
M.~Tautenhahn and I.~Veseli\'c, \emph{Sampling inequality for {$L^2$}-norms of
  eigenfunctions, spectral projectors, and {W}eyl sequences of {S}chr\"odinger
  operators}, J. Stat. Phys. \textbf{164} (2016), no.~3, 616--620.

\bibitem[TW09]{TucsnakW09}
M.~Tucsnak and G.~Weiss, \emph{Observation and control for operator
  semigroups}, Birkh\"auser, Basel, 2009.

\bibitem[Ves96]{Veselic-96}
I.~Veseli{\'c}, \emph{Lokalisierung bei zuf\"allig gest\"orten periodischen
  {S}chr\"odinger\-opera\-toren in {D}imension {E}ins}, Diplomarbeit,
  Ruhr-Universit\"at Bochum, 1996.

\bibitem[Ves08]{Veselic-08}
\bysame, \emph{Existence and regularity properties of the integrated density of
  states of random {S}chr\"odinger operators}, Lecture Notes in Mathematics,
  vol. 1917, Springer, Berlin, 2008.

\bibitem[Wei00]{Weidmann-00}
J.~Weidmann, \emph{Lineare {O}peratoren in {H}ilbertr\"aumen. {T}eil 1
  {G}rundlagen}, Teubner, Stuttgart, 2000.

\bibitem[WWZZ19]{WangWZZ}
G.~Wang, M.~Wang, C.~Zhang, and Y.~Zhang, \emph{Observable set, observability,
  interpolation inequality and spectral inequality for the heat equation in
  $\mathbb{R}^n$}, J. Math. Pures Appl. (9) \textbf{126} (2019), 144--194.

\bibitem[Zab08]{Zabczyk-08}
J.~Zabczyk, \emph{Mathematical control theory: An introduction}, Birkh\"auser,
  Boston, 2008.

\end{thebibliography}

\providecommand{\bysame}{\leavevmode\hbox to3em{\hrulefill}\thinspace}
\providecommand{\MR}{\relax\ifhmode\unskip\space\fi MR }
\providecommand{\MRhref}[2]{%
  \href{http://www.ams.org/mathscinet-getitem?mr=#1}{#2}
}
\providecommand{\href}[2]{#2}

\end{document}